\documentclass[11pt,reqno]{amsart}

\usepackage{epsfig}
\usepackage{color}
\usepackage{amsmath}
\usepackage{amssymb}

\newcommand {\ignore}[1]  {}

\newcommand{\sm}{\smallsetminus}

\newif\ifdraft\drafttrue

\draftfalse

\long\def\combarak#1{\ifdraft{\tt #1 }\else\ignorespaces\fi}

\font\sn = cmssi8 scaled \magstep0

\newcommand\name[1]{\label{#1}{\ifdraft{\sn [#1]}\else\ignorespaces\fi}}

\newcommand\eq[2]{{\ifdraft{\ \tt [#1]}\else\ignorespaces\fi}\begin{equation}\label{#1}{#2}\end{equation}}

\newcommand{\supp}{{\mathrm{supp}}}

\newcommand{\FF}{{\mathcal{F}}}
\newcommand{\GG}{{\mathcal{G}}}

\newcommand{\SSS}{{\mathcal {S}}}
\newcommand{\REL}{{\mathrm {REL}}}

\newcommand{\dist}{{\mathrm{dist}}}
\newcommand {\equ}[1]     {\eqref{#1}}

\newcommand{\R}{{\mathbb{R}}}
\newcommand{\Z}{{\mathbb{Z}}}

\newcommand{\HH}{{\mathcal{H}}}

\newcommand{\E}{{\mathbf{e}}}
\newcommand{\br}{{\mathbf{r}}}

\newcommand{\N}{{\mathbb{N}}}

\newcommand{\xx}{{\mathbf{x}}}
\newcommand{\yy}{{\mathbf{y}}}
\newcommand{\A}{{\bf{a}}}

\newcommand{\SL}{\operatorname{SL}}

\newcommand{\Res}{\operatorname{Res}}
\newcommand{\hol}{{\mathrm{hol}}}
\newcommand{\UU}{{\mathcal{U}}}

\newcommand{\interior}{{\rm int}}

\newcommand{\EE}{{\mathcal{E}}}

\newcommand{\QQ}{{\mathcal Q}}
\newcommand{\PP}{{\mathcal P}}
\newcommand{\LL}{{\mathcal L}}
\newcommand{\FFF}{{\mathcal F}}

\newcommand{\pp}{{\mathbf p}}
\newcommand{\bq}{{\mathbf q}}

\newcommand{\dev}{{\mathrm {dev}}}

\newcommand{\CC}{{\mathcal{C}}}

\newcommand{\vre}{\varepsilon}
\newcommand{\IE}{\mathcal{T}}
\newcommand{\B}{\mathbf{b}}
\newcommand{\cc}{\mathbf{c}}
\newcommand{\BB}{\mathcal{B}}
\newcommand{\BBB}{\mathbb{B}}
\newcommand{\AAA}{\mathbb{A}}

\newcommand\MM{{\mathcal M}}

\newcommand\PMF{{{\mathcal P}\kern-2pt\MM\FFF}}
\newcommand\PML{{{\mathcal P}\kern-2pt\MM\LL}}
\newcommand\til{\widetilde}
\newcommand\Mod{\operatorname{Mod}}

\newcommand\hhat{\widehat}

\catcode`\@=11
\long\def\@savemarbox#1#2{\global\setbox#1\vtop{\hsize\marginparwidth 
  \@parboxrestore\tiny\raggedright #2}}
\catcode`\@=12
\def\Empty{}
\newcommand\oplabel[1]{
  \def\OpArg{#1} \ifx \OpArg\Empty {} \else
        \label{#1}
  \fi}

\newlength{\saveu}

\newtheorem{thm}{Theorem}[section]
\newtheorem{lem}[thm]{Lemma}
\newtheorem{prop}[thm]{Proposition}
\newtheorem{Def}[thm]{Definition}
\newtheorem{cor}[thm]{Corollary}

\newtheorem{remark}[thm]{Remark}
\newtheorem{conj}{Conjecture}

\numberwithin{figure}{section}

\begin{document}

\title[Mahler's question for interval exchanges]{Cohomology classes
represented by
measured foliations, and Mahler's
question for interval exchanges}
\author{Yair Minsky}
\address{Yale University, New Haven, CT {\tt yair.minsky@yale.edu}}
\author{Barak Weiss}
\address{Ben Gurion University, Be'er Sheva, Israel 84105
{\tt barakw@math.bgu.ac.il}}
\date{\today}

\begin{abstract}
A translation surface on $(S, \Sigma)$ gives rise to two
transverse measured foliations $\FF, \GG$ on $S$ with singularities in
$\Sigma$, and by integration, to a pair of cohomology
classes $[\FF], \, [\GG] \in H^1(S, \Sigma; \R)$. Given a 
measured foliation $\FF$, we characterize the set of cohomology
classes $\B$ for which there is a measured foliation $\GG$ as above
with $\B = [\GG]$. This extends previous results of Thurston
\cite{Thurston stretch maps} and
Sullivan \cite{Sullivan}. 

We apply this to two problems: unique ergodicity of interval exchanges
and flows on the moduli space of translation surfaces. 
For a fixed permutation $\sigma \in \mathcal{S}_d$, the space $\R^d_+$
parametrizes the interval exchanges on $d$ 
intervals with permutation $\sigma$. We describe lines
$\ell$ in $\R^d_+$ such 
that almost every point in $\ell$ is uniquely ergodic. We also show that for
$\sigma(i) = d+1-i$, for almost every $s>0$, 
the interval exchange transformation corresponding to $\sigma$ and
$(s, s^2, \ldots, s^d)$ is uniquely ergodic. 
%
As another application we show that when $k=|\Sigma|
\geq 2,$ the 
operation of `moving the singularities horizontally' is globally
well-defined. 
We prove that there is a well-defined action of the group $B \ltimes
\R^{k-1}$ on the set of translation surfaces of type $(S, \Sigma)$
without horizontal saddle connections. Here $B \subset \SL(2,\R)$ is the
subgroup of upper triangular matrices. 
\end{abstract}

\maketitle
\section{Introduction}
\subsection{Motivating questions and nonsensical pictures}
To introduce the problems discussed in this paper, consider some
pictures. Suppose that $\A = (a_1, \ldots, a_d)$
is a vector with positive entries, $I = \left[0, \sum a_i \right)$ is an interval,
$\sigma$ is a permutation on $d$ symbols, and $\IE=\IE_{\sigma}(\A): I \to I$ is the
interval exchange obtained by cutting up $I$ into segments of lengths
$a_i$ and permuting them according to $\sigma$. A fruitful technique
for studying the dynamical properties of $\IE$ is to consider it as
the return map to a transverse segment along the vertical foliation in
a translation surface, i.e. a union of polygons with edges glued
pairwise by translations. See Figure \ref{figure: Masur} for an example with one polygon;
note that the interval exchange determines the 
horizontal coordinates of vertices, but there are many
possible choices of the vertical coordinates.

\begin{figure}[htp] 
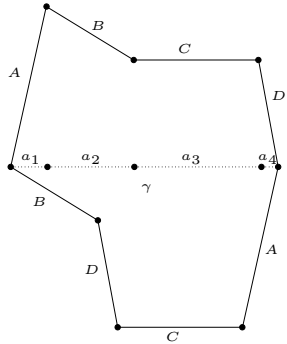
\caption{Masur's construction: an
interval exchange embedded in a one-polygon translation surface.}\name{figure: Masur}
\end{figure}

Given a translation surface $q$ with a transversal, one may deform it by
applying the horocycle flow, i.e. deforming the polygon with the linear map 
\eq{eq: horocycle matrix}{
h_s = \left( \begin{matrix} 1 & s \\ 0 & 1 \end{matrix} \right).
}
The return map to a transversal in $h_sq$ depends on $s$, so we
get a one-parameter family $\IE_s$ of interval exchange
transformations (Figure \ref{figure: horocycle polygon}). For
sufficiently small $s$, one has $\IE_s = \IE_\sigma(\A(s)),$ where
$\A(s) = \A + s\B$ is a line segment, whose derivative $\B = (b_1,
\ldots, b_d)$ is determined
by the {\em heights} of the vertices of the polygon. We will consider
an inverse problem: given a line segment $\A(s) = \A+s\B$, does there exist a
translation surface $q$ such that for all sufficiently small $s$,
$\IE_\sigma(\A(s))$ is the return map along vertical leaves to a
transverse segment in $h_sq$? Attempting to interpret this question with pictures,
we see that some choices of $\B$ lead to a translation
surface while others lead to nonsensical pictures -- see Figure
\ref{figure: nonsense}. The solution to this problem is given by 
Theorem \ref{thm: sullivan2}.

\begin{figure}[htp] 
\includegraphics{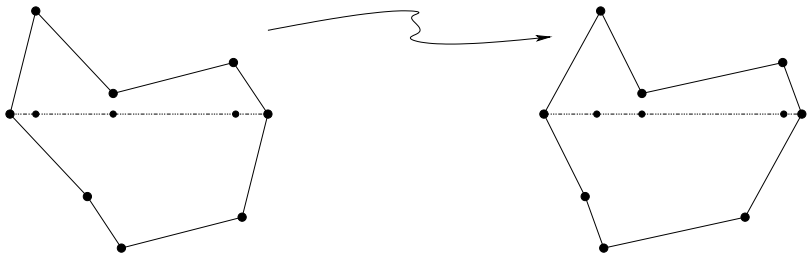}
\caption{The horocycle action gives a linear one parameter family of interval
exchanges} \name{figure: horocycle polygon}
\end{figure}

\begin{figure}[htp] 
\centerline{\hfill \includegraphics[scale=0.8]{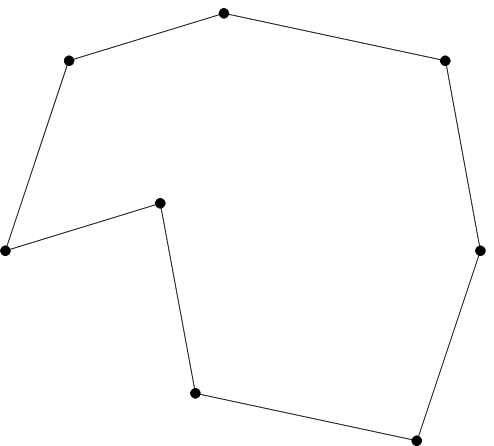}  \hfill
\includegraphics[scale=0.8]{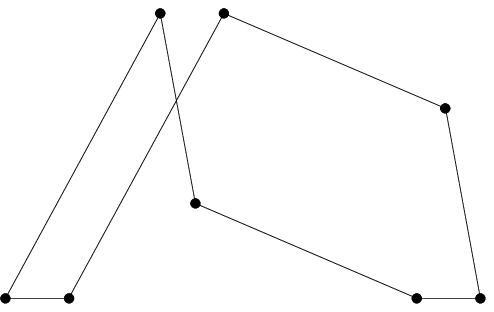}   \hfill}
\caption{The choice $\B = (2, 1, -1, -2)$
(left) gives a translation surface but 
what about $\B=(0, 3, -1, -2)$?} \name{figure: nonsense}
\end{figure}

Now consider a translation surface $q$ with two 
singularities. We may consider the operation of moving one singularity
horizontally with respect to the other. That is, at time $s$, the line
segments joining one singularity to the other are made longer by $s$,
while line segments joining a singularity to itself are unchanged. For
small values of $s$, one obtains a new translation surface $q_s$ by
examining the picture. But for large values of $s$, some of the
segments in the figure cross each other and it is not clear whether
the operation defined above gives rise to a well-defined surface. Our
Theorem \ref{thm: real rel main} shows that the operation of moving the zeroes is
well-defined for all values of $s$, provided one rules out the obvious
obstruction that two singularities connected by a horizontal segment
collide.

\begin{figure}[htp] \name{figure: real rel}
\centerline{\hfill \includegraphics[scale=0.45]{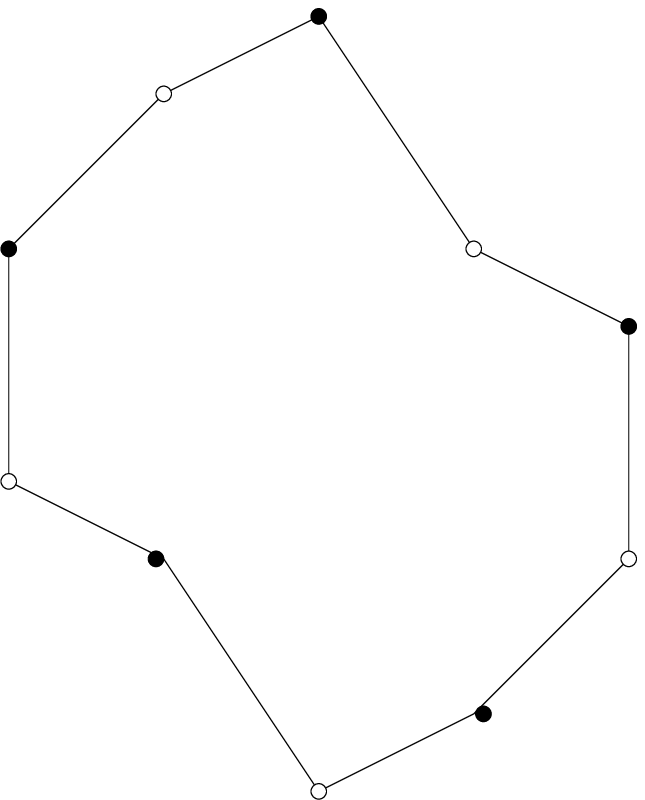}  \hfill
\includegraphics[scale=0.45]{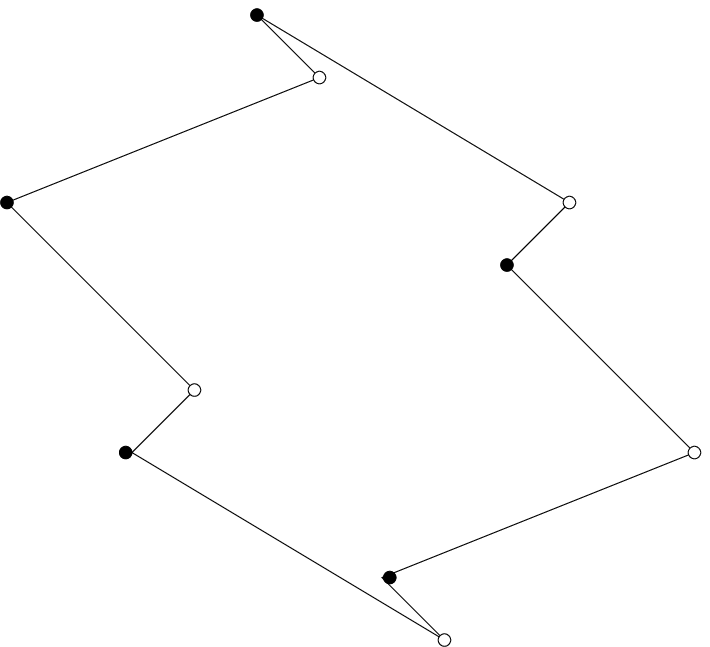}   \hfill \includegraphics[scale=0.45]{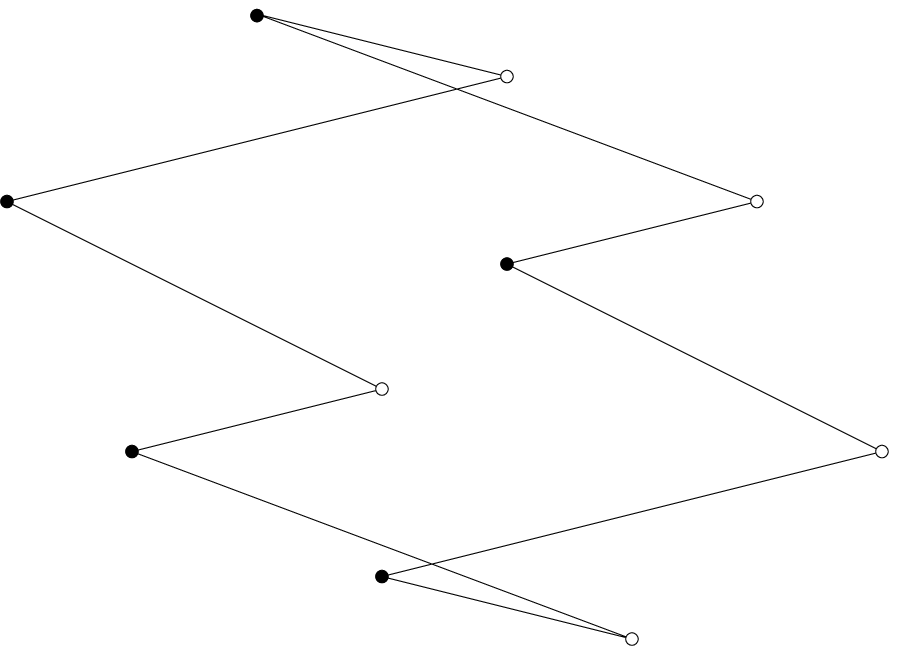}   \hfill}
\caption{The singularity $\circ$ is moved to the right with
respect to $\bullet$, and the
picture becomes nonsensical.}
\end{figure}

\subsection{Main geometrical result}
Let $S$ be a compact oriented surface of genus $g \geq 2$ and $\Sigma
\subset S$ a finite subset. A {\em
translation surface structure} on $(S,\Sigma)$ is an atlas of charts into the
plane, whose domains cover
$S \sm \Sigma$, and such that the
transition maps are translations. Such structures
arise naturally in complex 
analysis and in the study of interval exchange transformations and
polygonal billiards and have been the subject of intensive research, 
see the recent surveys  \cite{MT, Zorich survey}.

Several geometric structures on the plane can
be pulled back to $S \sm \Sigma$ via the atlas, among them the foliations of the
plane by horizontal and vertical lines. We call the resulting
oriented foliations of $S \sm \Sigma$ the {\em horizontal and vertical
foliation} 
respectively. Each can be completed to a singular foliation on $S$,
with a pronged {\em singularity} at each point of $\Sigma$. Label the
points of $\Sigma$ by $\xi_1, \ldots, \xi_k$ and fix natural 
numbers $r_1, \ldots, r_k $. We say that
the translation surface is {\em of
type $\mathbf{r} = (r_1, \ldots, r_k)$}
if the horizontal and vertical foliations have a $2(r_j+1)$-pronged
singularity at each $\xi_j$.

By pulling back $dy$ (resp. $dx$) from the plane, the
horizontal (vertical) foliation arising from a translation 
surface structure is equipped with a {\em transverse measure},
i.e. a
family of measures on each arc transverse to the foliation which is
invariant under holonomy along leaves. We will call an oriented singular foliation
on $S$, with singularities in 
$\Sigma$, which is equipped with a transverse
measure a {\em measured foliation on $(S, \Sigma)$}. We caution the reader that we
deviate from the convention adopted in several
papers on this subject, by considering the number and orders of
singularities as 
part of the structure of a measured foliation; we call these the {\em
type} of the foliation. In other words, we do not consider two measured
foliations which differ by a Whitehead move to be the same. 

Integrating the transverse measures gives rise to two
well-defined cohomology classes in the relative cohomology group
$H^1(S, \Sigma; \R)$. That is we obtain a map 
$$\mathrm{hol} : \left\{ \mathrm{translation\ surfaces\ on\ }
(S,\Sigma) \right \} \to \left(H^1(S, \Sigma; \R)\right)^2.
$$
This map is a local homeomorphism and serves to give coordinate
charts to the set of translation surfaces (see \S \ref{subsection: strata} for
more details), but it is not globally injective. For example,
precomposing with a homeomorphism which acts trivially on homology may
change a marked translation surface structure but does not change its
image in under hol; see \cite{McMullen American J} for more examples.
On the other hand it is not hard to see that the pair of horizontal
and vertical measured foliations associated to a translation surface
uniquely determine it, and hence the question arises of reconstructing
the translation surface from just the cohomological data recorded by
hol. Our main theorems give results in this direction. 

To state them we define the set of {\em (relative) cycles carried
by $\FF$}, denoted $H_+^\FF$, to be the image in $H_1(S,
\Sigma; \R)$ of all (possibly atomic) transverse measures on $\FF$
(see \S\ref{subsection: cohomology}). 

\begin{thm}
\name{thm: sullivan1}
Suppose $\FF$ is a measured foliation on $(S, \Sigma)$, and $\B \in H^1(S, \Sigma;
\R)$. Then the following are equivalent:
\begin{enumerate}
\item
There is a measured foliation $\GG$ on $(S, \Sigma)$, everywhere
transverse to $\FF$ and of the same type, such that $\GG$ represents $\B$. 
\item
Possibly after replacing $\B$ with $-\B$, for any $\delta \in H_+^{\FF}$, $\B\cdot \delta >0$.  
\end{enumerate}
\end{thm}

After proving Theorem \ref{thm: sullivan1} we learned from F. Bonahon
that it has a long and interesting history. Similar result were proved
by 
Thurston \cite{Thurston stretch maps} in the context of train tracks
and measured laminations, and by Sullivan \cite{Sullivan} in a very
general context involving foliations. Bonahon neglected to mention his
own contribution \cite{Bonahon}. Our result is a `relative
version' in that we control the type of the foliation, 
and need to be careful with the singularities. This explains why our
definition of $H_+^\FF$ includes the relative cycles carried by
critical leaves of $\FF$. The
proof we present here is 
close to the
one given by Thurston. 

The arguments proving Theorem \ref{thm: sullivan1} imply the following
stronger statement (see \S \ref{section: prelims} for detailed
definitions):

\begin{thm}\name{thm: hol homeo}
Given a topological singular foliation $\FF$ on $(S, \Sigma)$, let $\til \HH
(\FF)$ denote the set of marked translation surfaces whose vertical
foliation is topologically equivalent to $\FF$. Let $\AAA(\FF) \subset
H^1(S, \Sigma; \R)$ denote the set of cohomology classes corresponding
to (non-atomic) transverse measures on $\FF$. Let $\BBB(\FF) \subset
H^1(S, \Sigma; \R)$ denote the set of cohomology classes that pair
positively with all elements of $H^\FF_+ \subset H_1(S, \Sigma)$. Then 
$$
\mathrm{hol} : \til \HH (\FF) \to \AAA(\FF) \times \BBB(\FF)
$$
is a homeomorphism. 
\end{thm}

\subsection{Applications}
We present two applications of Theorem \ref{thm: sullivan1}. The first
concerns the generic properties of interval exchange transformations.
Let $\sigma$ be a permutation on
$d$ symbols and let $\R^d_+$ be the vectors $\A = (a_1, \ldots, a_d)$
for which $a_i>0, \, i=1, \ldots, d.$ The pair $\A, \sigma$ determines
an {\em interval exchange} $\IE_\sigma(\A)$ by subdividing the 
interval $I_\A=\left[ 0, \sum a_i \right)$ into  
$d$ subintervals of lengths $a_i$, which are
permuted according to $\sigma$. In 1982 
Masur \cite{Masur-Keane} and Veech \cite{Veech-zippers} confirmed a
conjecture of Keane, proving (assuming that $\sigma$ is irreducible) that almost
every $\A$, with respect to Lebesgue measure on $\R^d_+$, is {\em uniquely
ergodic}, i.e. the only invariant measure for $\IE_\sigma(\A)$ is
Lebesgue measure. On the other hand Masur and Smillie
\cite{MS} showed that the set of non-uniquely ergodic interval exchanges
is large in the sense of Hausdorff dimension. 
A basic problem is to understand the finer structure of the set of
non-uniquely ergodic interval exchanges. 
Specifically, motivated by analogous developments in diophantine
approximations, we will ask: For which curves
$\ell \subset \R^d_+$ is the non-uniquely ergodic set of zero
measure, with respect to the natural measure on the curve?
Which properties of a measure $\mu$ on $\R^d_+$
guarantee that $\mu$-a.e. $\A$ is uniquely
ergodic?

\ignore{
In a celebrated paper \cite{KMS}, Kerckhoff, Masur and Smillie 
solved an important special case of this problem and introduced a
powerful dynamical technique. Given a rational angled polygon 
$\mathcal{P}$, they 
considered the interval exchange $\IE_{\mathcal{P}}(\theta)$ which is
obtained from the 
billiard flow in $\mathcal{P}$ in initial direction $\theta$ as the
return map to a subset of $\mathcal{P}$, and  
proved  that for every $\mathcal{P}$, for almost every
$\theta$ (with respect to Lebesgue 
measure on $[0, 2\pi]$), $\IE_\mathcal{P}(\theta)$ is uniquely
ergodic. The result was later extended by Veech \cite{Veech-fractals},
who showed that Lebesgue measure 
can be replaced by any measure on $[0, 2\pi]$ satisfying a certain
decay condition, for
example the coin-tossing measure on Cantor's 
middle thirds set. 
The above results correspond to the
statement that for a certain $\sigma$ (depending on $\mathcal{P}$),
$\mu$-a.e. $\A \in \R^d_+$ is uniquely ergodic for 
certain measures $\mu$ supported on line segments in $\R^d_+$. 
}

In this paper we obtain several results in this direction, involving
three ingredients: a curve in $\R^d_+$; a
measure supported on the curve; and a dynamical property of interval
exchanges. The goal will be to 
understand the dynamical properties of points in the support of the
measure. We state these results, and some open questions in this
direction, in \S \ref{section: results on iets}. To illustrate them we state
a special case,
which may be thought of as an interval exchanges analogue of a
famous
result of Sprindzhuk (Mahler's conjecture, see e.g. 
\cite[\S4]{dimasurvey}): 
\begin{thm}
\name{thm: mahler curve}
For $d \geq 2,$ let 
\eq{eq: mahler curve}{
\A(x) = \left(x, x^2, \ldots, x^d\right)
}
and let
$\sigma(i)= d+1-i$. Then $\A(x)$ is uniquely ergodic 
for Lebesgue a.e. $x>0$. 
%
\end{thm}
Identifying two translation structures which differ by a
precomposition with an orientation preserving homeomorphism of $S$
which fixes each point of $\Sigma$ we obtain the {\em stratum
$\HH(\mathbf{r})$ of
translation surfaces of type $\mathbf{r}$}. 
There is an action of $G = \SL(2,\R)$ on $\HH(\mathbf{r})$, and its
restriction to the subgroup $\{h_s\}$ as in \equ{eq: horocycle matrix}
is called the {\em horocycle flow}.  
To prove our results on unique ergodicity we employ the strategy, 
introduced in 
\cite{KMS}, of lifting interval exchanges to translation
surfaces, and studying the
dynamics of the $G$-action on $\HH(\mathbf{r})$. Specifically we use
quantitative
nondivergence estimates \cite{with Yair} for the horocycle flow. Theorem \ref{thm: sullivan1} is used
to characterize the lines in $\R^d_+$ which may be lifted to horocycle
paths. 

\medskip

The second application concerns an operation of `moving singularities
with respect to each other' which has been discussed in the literature
under various names (cf. \cite[\S9.6]{Zorich survey}) and which we
now define. Let 
$\til \HH(\mathbf{r})$ be the {\em stratum of marked translation
surfaces of type $\mathbf{r}$}, i.e. two translation surface structures are equivalent 
if they differ by precomposition by a
homeomorphism of $S$ which fixes $\Sigma$ and is isotopic to the
identity rel $\Sigma$. Integrating transverse measures as above induces
a well-defined map $\til \HH(\mathbf{r}) \to H^1(S, \Sigma; \R^2)$ which
can be used to endow $\til \HH(\mathbf{r})$ (resp. $\HH(\mathbf{r})$)
with the structure of an affine manifold (resp. orbifold), such that
the natural map $\til \HH(\mathbf{r}) \to 
\HH(\mathbf{r})$ is an orbifold cover. We describe foliations
on $\til \HH(\mathbf{r})$ which descend to well-defined foliations on
$\HH(\mathbf{r})$. The two summands in the splitting
\eq{eq: splitting}{ 
H^1(S, \Sigma; \R^2) \cong H^1(S, \Sigma; \R) \oplus H^1(S, \Sigma;
\R)}
induce two foliations on $\til \HH(\mathbf{r})$, which we call the {\em
real foliation} and {\em imaginary foliation} respectively. Also, considering the
exact sequence in cohomology 
\eq{eq: defn Res}{
H^0(S;\R^2) \to H^0(\Sigma ; \R^2) \to H^1(S, \Sigma ; \R^2)
\stackrel{\mathrm{Res}}{\to} H^1(S; \R^2) \to \{0\},
}
we obtain a natural subspace $\ker \mathrm{Res} \subset H^1(S, \Sigma;
\R^2)$, consisting of the cohomology classes which
vanish on the subspace of `absolute periods' $H_1(S) \subset
H_1(S, \Sigma).$ 
The foliation induced on $\til \HH(\mathbf{r})$ is called the
{\em REL foliation} or {\em kernel foliation.} Finally, intersecting
the leaves of the real 
foliation with those of the REL foliation yields the {\em
real REL foliation}.  
It has leaves of dimension $k-1$ (where $k=|\Sigma|$). Two nearby translation
surfaces $q$ and $q'$ are in the same plaque if the integrals of the flat structures along all 
closed curves are the same on $q$ and $q'$, and if the integrals of
curves joining {\em distinct} singularities only differ in their
horizontal component.  
Intuitively, $q'$ is obtained from $q$ by fixing one singularity as a
reference point and moving the other singularities
horizontally. Understanding this foliation is
important for the study of the dynamics of the horocycle flow. It was
studied in some restricted settings in \cite{EMM, CW}, where it was
called {\em Horiz}.  

The leaves of the kernel foliation, and hence the real REL foliation,
are equipped with a natural translation structure, modeled on the
vector space $\ker \, \mathrm{Res} \cong H^0(\Sigma; \R)/H^0(S, \R)$. One
sees easily that the leaf of $q$ is incomplete if, when moving along
the leaf, a saddle connection
on $q$ is made to have length zero, i.e. if `singularities
collide'. Using Theorems \ref{thm: sullivan1} and \ref{thm: hol homeo} we 
show in Theorem \ref{thm: real rel main} that this is the only
obstruction to completeness of leaves. This implies that on a
large set, the leaves of real REL are the orbits of an action. More
precisely, let $\QQ$ be the set of translation surfaces
with no horizontal saddle connections, in a finite cover $\hat{\HH}$ of
$\HH(\br)$ (we take a finite cover to make $\HH(\br)$ into a manifold). This is a set
of full measure which is invariant under the group $B$ of upper
triangular matrices in $G$. We show that it coincides with the set of
complete real REL leaves. Let $F$ denote the group $B \ltimes
\R^{k-1}$, where $B$ acts on $\R^{k-1}$ 
via 
$$\left( \begin{matrix} a & b \\ 0 & 1/a \end{matrix} \right) \cdot \vec{v} = a\vec{v}.$$ 
We prove:
\begin{thm}
\name{thm: real rel action}
The group $F$ acts on $\QQ$ continuously and affinely, preserving the
natural measure, and leaving invariant the subset of translation surfaces
of area one. The action of $B$ is the same as that obtained by restricting the
$G$-action, and the $\R^{k-1}$-action is transitive on each real REL leaf in
$\QQ$. 
\end{thm}
Note that while the $F$-action is continuous, $\QQ$ is not
complete: it is the complement in $\hat{\HH}$ of a dense
countable union of proper affine submanifolds with boundary. Also note
that the leaves of the real foliation or the kernel foliation
are not orbits of a group action on $\hat{\HH}$ --- but see
\cite{EMM} for a related discussion of pseudo-group-actions.

\subsection{Organization of the paper}
We present the proof of Theorem \ref{thm:
sullivan1} in \S3 and of Theorem \ref{thm: hol homeo}, in \S
\ref{section: dev homeo}. We interpret these theorems in the language of
interval exchanges in \S5. This interpretation furnishes a link
between line segments in the space of interval exchanges, and
horocycle paths in a corresponding stratum of translation surfaces: it
turns out that the line segments which may be lifted to horocycle
paths form a cone in the tangent space to interval exchange space, and
this cone can be explicitly described in terms of a bilinear form
studied by Veech. 
We begin \S6 with a brief discussion of Mahler's question in
diophantine approximation, and the question it motivates for interval
exchanges. We then state in detail our results for
generic properties of interval exchanges. The proofs of these results occupy
\S7--\S10. Nondivergence results for horocycles make it possible to
analyze precisely the properties of interval exchanges along a
line segment, in the cone of directions described in \S \ref{section: REL}. To obtain
information about curves we approximate them by line segments, and
this requires the quantitative nondivergence results obtained in
\cite{with Yair}. 
In \S11 we prove our results concerning real REL. These sections
may be read independently of \S6--\S10. We conclude with a discussion
which connects real REL with some of the objects encountered in
\S7-\S10.

\subsection{Acknowledgements}
We thank John Smillie for many valuable discussions.
We thank Francis Bonahon for pointing out the connection between our
Theorem \ref{thm: sullivan1} and previous work of Sullivan and
Thurston. The
authors were supported by  
BSF grant 2004149, ISF grant 584/04 and NSF grant DMS-0504019.

\section{Preliminaries}\name{section: prelims}
In this section we recall some standard facts and set our
notation. For more information we refer the reader to 
\cite{MT, Zorich survey} and the references therein. 
\subsection{Strata of translation surfaces}\name{subsection: strata}
Let $S, \, \Sigma = (\xi_1, \ldots, \xi_k), \, \mathbf{r}=(r_1,
\ldots, r_k)$ be as in
the introduction. A {\em translation 
structure} (resp., a {\em marked translation structure}) of type
$\mathbf{r}$ on $(S, \Sigma)$
is an equivalence class of $(U_{\alpha}, \varphi_{\alpha})$, where:
\begin{itemize}
\item
$(U_{\alpha}, \varphi_{\alpha})$ is an atlas of charts for $S\sm
\Sigma$;
\item
the transition functions $\varphi_{\beta} \circ \varphi^{-1}_{\alpha}$
are of the form 
$$\R^2 \ni \vec{x} \mapsto \vec{x}+c_{\alpha, \beta};$$
\item
around each $\xi_j \in
\Sigma$ the charts glue together to form a cone point with cone angle 
$2\pi(r_j+1)$. 

\end{itemize}
By definition  $(U_{\alpha}, \varphi_{\alpha}), \ (U'_{\beta},
\varphi'_{\beta})$ are equivalent if there is an orientation preserving  
homeomorphism $h: S\to S$ (for a marked structure, isotopic to the
identity via an isotopy fixing $\Sigma$), fixing all points of $\Sigma$, such that 
$(U_{\alpha}, \varphi_{\alpha})$  
is compatible with  $\left(h(U'_{\beta}), \varphi'_{\beta} \circ h^{-1}\right).$
Thus the equivalence class $\bq$ of a marked translation surface is a subset
of that of the corresponding translation surface $q$, and we will say
that $\bq$ is obtained from $q$ by {\em specifying a marking}, $q$
from $\bq$ by {\em forgetting the marking}, and
write $q =\pi(\bq)$.
Note that our convention is that singularities are labelled.

Pulling back $dx$ and $dy$ from the coordinate charts we obtain two
well-defined closed 1-forms, which we can integrate along any path
$\alpha$ on $S$. If $\alpha$ is a cycle or has endpoints in $\Sigma$
(a relative cycle), then the result, which we denote by 
$$\hol(\alpha, \bq) = \left(\begin{matrix} x(\alpha, \bq) \\
y(\alpha, \bq) \end{matrix} \right) \in \R^2,$$
depends only on the homology class of $\alpha$ in $H_1(S, \Sigma)$. We
let $\hol(\bq) = \hol(\cdot, \bq)$ be the corresponding element of $H^1(S, \Sigma;
\R^2)$, with coordinates $x(\bq), y(\bq)$ in $H^1(S, \Sigma; \R)$. 

A {\em saddle connection} for $q$ is a straight segment which connects
singularities and does not contain singularities in its interior.

 The set of all (marked) translation surfaces on $(S,
\Sigma)$ of type $\mathbf{r}$ is called the {\em stratum of
(marked) translation surface of type $\mathbf{r}$} and is denoted by
$\HH(\mathbf{r})$ (resp. $\til \HH(\mathbf{r})$). We have suppressed
the dependence on $\Sigma$ from the notation since for a given type
$\mathbf{r}$ there is an
isomorphism between the corresponding set of translation surfaces on
$(S, \Sigma)$ and on $(S, \Sigma')$ for any other finite subset
$\Sigma' = \left(\xi'_1, \ldots, \xi'_k\right)$. 

The map $\hol: \til \HH \to H^1(S, \Sigma; \R^2)$ just defined gives
local charts for $\til \HH$, endowing it (resp. $\HH$) with the
structure of an affine manifold (resp. orbifold). 
To see how this works, fix a triangulation $\tau$ of $S$ with vertices
in $\Sigma$. Then $\hol(\bq)$ associates a vector in the plane to each
oriented edge in $\tau$, and hence associates an oriented Euclidean
triangle to each oriented triangle of $\tau$. If all the orientations
are consistent, then a translation structure with the same holonomy as
$\bq$ can be realized explicitly by gluing the Euclidean triangles to
each other. Let $\til \HH_\tau$ be the set of all translation
structures obtained in this way (we say that $\tau$ is {\em realized
  geometrically} in such a structure). Then the restriction $\hol:
\til \HH_{\tau} \to H^1(S, \Sigma; 
\R^2)$ is injective and maps onto an open subset. Conversely every
$\bq$ admits some geometric triangulation (e.g. a Delaunay
triangulation as in \cite{MS}) and hence $\HH$ is covered
by the $\HH_{\tau}$, and so these provide an atlas for a linear
manifold structure on $\til \HH$. We should remark that a topology on
$\til \HH$ can be defined independently of this, by considering nearly
isometric comparison maps between different translation structures,
and that this topology is the same as that induced by the charts of
hol.  

Let $\Mod(S, \Sigma)$ denote the mapping class group, i.e. the
orientation preserving homeomorphisms of $S$ fixing $\Sigma$
pointwise, up to an isotopy fixing $\Sigma$. 
The map hol is $\Mod(S, \Sigma)$-equivariant. This means that for
any $\varphi \in \Mod(S, \Sigma)$, $\hol(\bq \circ \varphi) =
\varphi_* \hol(\bq)$, which is nothing more than the linearity of the
holonomy map with respect to its first argument. 

One can show that the $\Mod(S, \Sigma)$-action on $\til \HH$ 
is properly discontinuous. Thus $\HH = \til \HH /\Mod(S, \Sigma)$ is a
linear orbifold and $\pi: \til \HH \to \HH$ is an orbifold covering
map. 
Since $\Mod(S, \Sigma)$ contains a finite index torsion-free subgroup
(see e.g. \cite[Chap. 1]{Ivanov}), there is a finite cover
$\hat{\HH} \to \HH$ such that $\hat{\HH}$ is a manifold, and we have
\eq{eq: dimension of HH}{
\dim \HH = \dim \til \HH = \dim \hat{\HH}= \dim H^1(S, \Sigma; \R^2)= 2(2g+k-1).
}
The Poincar\'e Hopf index theorem implies that 
\eq{eq: Gauss Bonnet}{\sum r_j = 2g-2.}

There is an action of $G = \SL(2,\R)$ on $\HH$ and on $\til \HH$ by post-composition on each
chart in an atlas. The projection $\pi : \til \HH \to \HH$ is $G$-equivariant. 
The $G$-action is linear in the homology coordinates, namely, given a
marked translation surface structure $\bq$ and $\gamma \in 
H_1(S, \Sigma)$, and given $g \in G$, we have
\eq{eq: G action}{
\hol(\gamma, g\bq) = g \cdot \hol(\gamma, \bq),
}
where on the right hand side, $g$ acts on $\R^2$ by matrix
multiplication. 

We will write 
\[
g_t = \left(\begin{array}{cc} e^{t/2} & 0 \\ 0 & e^{-t/2} 
\end{array}
\right), \ \ \ \ \ 
r_{\theta}=
\left(\begin{array}{cc}
\cos \theta & -\sin \theta \\
\sin \theta & \cos \theta
\end{array}
\right).
\]

\ignore{
\combarak{We may or may not the discussion of the entire moduli space
below, depending on what we can say 
about `nondivergence along REL'.}
For a fixed genus $g$, the different strata of genus $g$ glue together
to form the {\em moduli space of translation surfaces of genus
$g$}. This space, which we denote by $\Omega_g$, is
sometimes referred to as the {\em bundle of holomorphic $1$-forms over
moduli space.} It is also a non-compact orbifold in which the strata
are locally open sub-orbifolds. The $G$-action extends continuously to
an action on $\Omega_g$. The precise definition of the orbifold
topology on $\Omega_g$ will not play a role in this paper.

}

\subsection{Interval exchange transformations}
\name{subsection: iets}
Suppose $\sigma$ is a permutation on $d$ symbols. 
For each 
$$\A \in \R^d_+ = \left\{(a_1, \ldots, a_d) \in \R^d :
\forall i, \, a_i>0 \right \}$$
we have an interval exchange transformation $\IE_{\sigma}(\A)$ defined by
dividing the interval 
$$I=I_\A = \left[0, \sum a_i \right)$$
into subintervals of lengths $a_i$ and permuting them
according to $\sigma.$ It is
customary to take these intervals as closed on the left and open on
the right, so that the resulting map has $d-1$ discontinuities and is
left-continuous. 
 
More precisely, set 
$x_0=x'_0=0$ and for $i=1,
\ldots, d$, 
\eq{eq: disc points2}{
x_i=x_i(\A)
= \sum_{j=1}^i a_j, \ \ \
x'_i =x'_i(\A) 
= \sum_{j=1}^i a_{\sigma^{-1}(j)} = \sum_{\sigma(k)\leq i} a_k;
}
then for every $x \in I_i=[x_{i-1}, x_i)$ we have 
\eq{eq: defn iet}{
\IE(x)=\IE_{\sigma}(\A)(x)=x-x_{i-1}+x'_{\sigma(i)-1} = x-x_i+x'_{\sigma(i)}.
}
In particular, if (following Veech \cite{Veech-measures}) we let $Q$ be
the alternating bilinear form  
given by 
\eq{eq: defn Q1}{
Q(\E_i, \E_j) = \left\{\begin{matrix}1 && i>j, \, \sigma(i)<\sigma(j) \\ -1
&& i<j, \, \sigma(i) >\sigma(j) \\ 0 && \mathrm{otherwise}
\end{matrix} \right. 
}
where $\E_1, \ldots, \E_d$ is the standard basis of $\R^d$, 
then 
$$\IE(x) -  x= Q(\A, \E_i).
$$
An interval exchange $\IE: I \to I$ is said to be {\em minimal} if
there are no proper closed $\IE$-invariant subsets of $I$.
We say
that $\IE$ is {\em uniquely ergodic} if the only invariant measure for
$\IE$, up to scaling, is Lebesgue measure. 
We will say that $\A \in \R^d_+$ is minimal or uniquely ergodic if
$\IE_{\sigma}(\A)$ is.  

Below we will assume that 
$\sigma$ is {\em irreducible}, i.e. there is no $k<d$ such that
$\sigma$ leaves the subset $\{1, \ldots, 
k\}$ invariant, and {\em admissible} (in Veech's sense), see
\S \ref{subsection: transversal}.
For the questions about interval exchanges which we will
study, these hypotheses on $\sigma$ entail no loss of generality. 

It will be helpful to consider a more general class of maps which we
call {\em generalized interval exchanges.} Suppose $J$ is a finite
union of intervals. A generalized interval exchange $\IE: J \to J$ is
an orientation preserving  piecewise isometry of $J$, i.e. it is a map
obtained by subdividing $J$ into finitely many subintervals and
re-arranging them to obtain $J$. These maps are not often considered because
studying their dynamics easily reduces to studying interval
exchanges. However they will arise naturally in our setup.

\subsection{Measured foliations, transversals, and interval exchange
induced by a translation surface}\name{subsection: transversal}
Given a surface $S$ and a finite $\Sigma \subset S$, a 
{\em singular foliation (on $S$ with singularities in $\Sigma$)} is a
foliation $\FF$ on $S \sm \Sigma$ such that for any $z \in \Sigma$ there
is $k=k_z \geq 3$ such that $\FF$ extends to form a $k$-pronged
singularity at $z$. A singular foliation $\FF$ is orientable if there
is a continuous choice of a direction on each leaf. If $\FF$ 
is orientable then $k_z$ is even for all
$z$. Leaves which meet the singularities are called {\em critical.}  
A {\em transverse measure} on a singular foliation $\FF$ is a family
of measures defined on arcs transverse to the 
foliation and invariant under restriction to subsets and isotopy
through transverse arcs. A {\em measured foliation} is a singular
foliation equipped with a transverse measure, which we further require
has no atoms and has full support (no open arc has measure zero). We will only consider
orientable singular foliations which 
can be equipped with a transverse measure. This implies that the
surface $S$ is decomposed into finitely many domains on each of which
the foliation is either minimal (any ray is dense) or periodic (any
leaf is periodic). A periodic component is also known as a {\em
  cylinder}. These components are separated by saddle
connections.

Given a flat surface
structure $\bq$ on $S$, pulling back via charts the vertical and
horizontal foliations on $\R^2$ give oriented singular foliations on
$S$ called the vertical and horizontal foliations,
respectively. Transverse measures are defined by integrating the
pullbacks of $dx$ and $dy$, i.e. they correspond to the holonomies
$x(\bq)$ and $y(\bq)$. Conversely, given two oriented everywhere
transverse measured foliations 
on $S$ with singularities in $\Sigma$, one obtains an atlas of charts 
as in \S \ref{subsection: strata} by integrating the measured
foliation. I.e., for each $z \in S \sm \Sigma$, a local coordinate
system is obtained by taking a simply connected neighborhood $U
\subset S \sm \Sigma$ of $z$ and defining the two coordinates of
$\varphi(w) \in \R^2$ to be the integral of the measured foliations
along some path connecting $z$ to $w$ (where the orientation of the
foliations is used to determine the sign of the integral). One can
verify that this procedure produces an atlas with the required
properties.

We will often risk confusion by using the 
symbol $\FF$ to denote both a measured foliation and the
corresponding singular foliation supporting it. A singular
foliation is called {\em minimal} if any noncritical leaf is dense,
and {\em uniquely ergodic} if there is a unique (up to scaling)
transverse measure on $S$ which is supported on noncritical
leaves. Where confusion is unavoidable we say that  
$q$ is minimal or uniquely ergodic if its vertical foliation is. 

At a singular point $p \in \Sigma$ with $k$ prongs, a small
neighborhood of $p$ divides into $k$ foliated disks, glued along
leaves of $\FF$, which we call {\em foliated half-disks}. A foliated
half-disk is either contained in a single periodic component or in a
minimal component. 

Now let $\FF$ be a singular foliation on a surface $S$ with
singularities in $\Sigma$. We will consider three kinds of
transversals to $\FF$. 
\begin{itemize}

\item
We define a {\em transverse system} to be an injective map
$\gamma: J \to S$ where $J$ is a finite union of intervals $J_i$, the
restriction of $\gamma$ to each interval $J_i$ is a smooth embedding,
the image 
of the interior of $\gamma$ intersects every non-critical leaf of $\FF$ 
transversally, and does not intersect $\Sigma$. 
\item
We define a {\em
judicious curve} to be a transverse system $\gamma: J \to S$ with $J$
connected, such that $\gamma$ begins and ends at
singularities, and the interior of $\gamma$ intersects all leaves
including critical ones.
\item
We say a transverse system $\gamma$ is {\em special} if all its
components are of the following types (see Figure \ref{figure:
  special}): 
\begin{itemize}
\item For every foliated half-disk $D$ of a singularity $p \in
  \Sigma$ which is contained in a minimal component, there is a
  component of $\gamma$ whose interior intersects $D$ and terminating
  at $p$. This component of $\gamma$ meets $\Sigma$ at only one
  endpoint. 
\item
For every periodic
component (cylinder) $P$ of $\FF$, $\gamma$ contains one arc crossing $P$
and joining two singularities on opposite sides of $P$. 

\end{itemize}
\end{itemize}

\begin{figure}[htp] \name{figure: special}
\center{\includegraphics{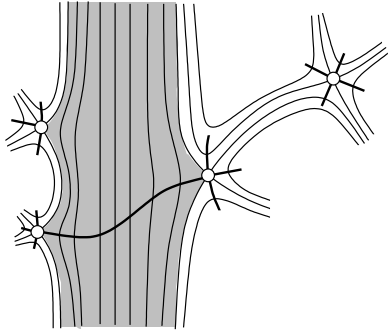}}
\caption{A {\em special} transverse system cuts across periodic components
and into minimal components.}
\end{figure}

Note that since every non-critical leaf in a minimal component is
dense, the non-cylinder edges of a special transverse system can be
made as short as we like, without destroying the property that they
intersect every non-critical leaf.

In each of these cases, we can
parametrize points of $\gamma$ using the transverse measure, and
consider the first return map to $\gamma$ when moving up along
vertical leaves. When $\gamma$ is a judicious curve, this
is an interval exchange transformation which we denote by $\IE(\FF, \gamma)$,
or by $\IE(q, \gamma)$ when $\FF$ is the vertical foliation of
$q$. Then there is a unique choice of $\sigma$ and $\A$ with $\IE(\FF,
\gamma) = \IE_{\sigma}(\A)$, and with $\sigma$ an irreducible
admissible permutation. The corresponding number of intervals is 
\eq{eq: for dim count}{
d=2g+|\Sigma|-1;
}
note that $d = \dim H^1(S, \Sigma; \R) = \frac12 \dim \HH$ if $q \in
\HH$.
The return
map to a transverse system is a generalized interval exchange. We
denote it also by $\IE(\FF, \gamma)$. In each of the 
above cases, any non-critical leaf returns to $\gamma$ infinitely many
times. 

If $\gamma$ is a transversal to $\FF$ then
$\IE(\FF, \gamma)$ completely determines the transverse measure on
$\FF.$ In particular the vertical foliation of $q$ is 
uniquely ergodic (minimal) if and only if $\IE(q, \gamma)$ is for some
(any) transverse system $\gamma$. 

There is an inverse construction which associates with an irreducible
permutation $\sigma$ a surface $S$ of genus $g$, and a $k$-tuple $\br$
satisfying \equ{eq: Gauss Bonnet}, such that the 
following holds. For any $\A \in \R^d_+$ there is a translation
surface structure $q \in \HH(\br)$, and a transversal $\gamma$ on
$S$ such that $\IE_\sigma(\A) = \IE(q, \gamma).$ Variants of this
construction can be found in \cite{ZK, Masur-Keane,
Veech-zippers}. Veech's admissibility condition amounts to requiring
that there is no transverse arc on $S$ for which $\IE(q, \gamma)$
has fewer discontinuities. 
Fixing $\sigma$, we say that a flat structure $q$ on $S$ is {\em a
lift} of $\A$ if there is a 
judicious curve $\gamma$ on $S$ such that $\IE(q,\gamma) =
\IE_\sigma(\A)$. It is known that for any $\sigma$, there is a stratum
$\HH$ such that all lifts of all $\A$ lie in $\HH$. We call it the
{\em stratum corresponding to $\sigma$.}

\ignore{

\subsection{Constructing a translation surface from an interval
exchange}\name{subsection: admissible}
There is an inverse construction which associates with an irreducible, admissible
permutation $\sigma$ a surface $S$ of genus $g$, and a $k$-tuple $r_1,
\ldots, r_k$ satisfying \equ{eq: Gauss Bonnet}, such that the
following holds. For any $\A \in \R^d_+$ there is a translation
surface structure $q \in \HH(r_1, \ldots, r_k)$, and a transversal $\gamma$ on
$S$ such that $\IE_\sigma(\A) = \IE(q, \gamma).$ Variants of this
construction can be found in \cite{ZK, Masur-Keane,
Veech-zippers}. Masur's construction will be convenient for us in the
sequel, and we recall it now. 

Set
\eq{eq: Masur choice}{b_i = \sigma(i) - i,}
and in analogy with \equ{eq: disc points2}, set $y_0=y'_0=0,$ and for 
$i=1, \ldots, d$, 
\eq{eq: disc points1}{
y_i=y_i(\B)= \sum_{j=1}^i b_j, \ \ \
y'_i =y'_i(\B) = \sum_{j=1}^i b_{\sigma^{-1}(j)} =
\sum_{\sigma(k)\leq i} b_k.
}
Then it is immediate from the irreducibility of $\sigma$ that $y_i>0$
and $y'_i<0$ for  $i=1, \ldots, d-1$. 
Set $x_i=x_i(\A),\
x'_i = x'_i(\A)$ as in \equ{eq: disc points2},
and 
\eq{eq: points in plane}{
P_i = (x_i, y_i), \ \  P'_i = (x'_i, y'_i).}
By the above the $P_i$ (resp. $P'_i$) lie above (resp. below) the
x-axis, and $P_0=P'_0, P_d=P'_d.$
Let $f$ (resp. $g$) be the piecewise linear functions
whose graph is the curve passing through all the points $P_i$
(resp. $P'_i$), and let 
\eq{eq: defn mathcalP}{
\mathcal{P} = \left\{(x,y)
: x \in I,\, g(x) \leq y \leq f(x) \right 
\}.
}
Then $\mathcal{P}$ is the closure of 
a connected domain with polygonal boundary, see Figure \ref{figure: Masur}.
We may identify $(x,f(x))$ in the
upper boundary to $(\IE x,g(\IE x))$ in the lower
boundary to obtain a translation surface $q$. These
identifications induce an equivalence relation on 
$\{P_0, \ldots, P_d, P'_0, \ldots, P'_d\}$, and one may compute
\combarak{the computation is in the tex file if we want to say more} 
the total angle around each equivalence class. By definition, $\sigma$ is
{\em admissible in Veech's sense} if the total angle around each is more than $2\pi$,
so that each $P_i$ and $P'_j$ is identified with a singularity. 
We have explicitly constructed an atlas of charts for a flat structure
$q$ 
on a surface of genus $g$, where $g$ is computed via \equ{eq: Gauss
Bonnet} from the angles around singularities. 
\ignore{
To compute these
equivalence classes and the 
total angle around each, suppose for $i \in \{1, \ldots, d-1\}$ that we are given a
downward-pointing vector based at $P_i$. If we rotate
this vector through an angle of $2\pi$, using boundary identifications so
that the vector is always pointing into the interior of $\mathcal{P}$,
we obtain a downward-pointing vector based at $P_{\til \sigma(i)}$,
where $\til \sigma : \{1, 
\ldots, d-1\} \to \{1, \ldots, d-1\}$ is an auxilliary permutation
defined by
$$\til
\sigma(i) = \left\{ \begin{matrix}
\sigma^{-1}(\sigma(1)-1) && \sigma(i+1) =1 \\
\sigma^{-1}(d) && \sigma(i+1) = \sigma(d)+1 \\
 \sigma^{-1}(\sigma(i+1)-1)  &&
\mathrm{otherwise}
 \end{matrix} \right.
$$
(note this differs slightly from the auxilliary permutation defined in
\cite{Veech-zippers}). The singularities are the equivalence classes
of $P_i$ under the relation generated by $i \sim \til \sigma(i)$ and
$0 \sim \sigma^{-1}(1)-1, d \sim \sigma^{-1}(d).$ The total angle
around the a singularity is $2\pi k$ where $k$ is the length of the
corresponding cycle for $\til \sigma$. We will say that
$\sigma$ is {\em admissible} if all cycles for $\til
\sigma$ contain at least two elements, i.e., $r_j$ as in \S 2.2
satisfy  $r_j\geq 1$ for
all $j$.

Thus $q_0 \in \HH = \HH(\mathbf{r})$ where $\mathbf{r}=(r_1, \ldots,
r_k)$. 
, and in particular, by \equ{eq: Gauss Bonnet} we have
constructed a topological surface $S$ of genus at least 2. 
}
Let $\gamma$ be a path passing from left to right through the
interior of
$\mathcal{P}$. 
By construction $\gamma$ is judicious for $q$, and 
$\IE(q, \alpha) = \IE_{\sigma}(\A_0)$.

\begin{figure}[htp] \name{figure: Masur}
\input{masur.pstex_t}
\caption{Masur's construction: a flat surface constructed from an
interval exchange via a polygon.}
\end{figure}

Fixing $\sigma$, we say that a flat structure $q$ on $S$ is {\em a lift} of $\A$ if there is a
judicious curve $\gamma$ on $S$ such that $\IE_\sigma(q,\gamma) =
\IE_\sigma(\A)$. It is known that for any $\sigma$, there is a stratum
$\HH$ such that all lifts of all $\A$ lie in $\HH$. We call it the
{\em stratum corresponding to $\sigma$.}
\ignore{
Clearly $\mathcal{P}$, with the edge identifications described above,
gives a cellular decomposition of $S$ containing the points of $\Sigma$
as vertices (0-cells), which we will continue by $\mathcal{P}$. In
particular there is a natural identification of $H^1(S, \Sigma; \R)$ with
$H^1(\mathcal{P}, \Sigma; \R)$. It is easy to see (and also follows
via \equ{eq: for dim count} by a dimension count) that the edges of $\mathcal{P}$ are
linearly independent cycles in $H_1(\mathcal{P}, \Sigma; \R)$. Thus
there is no distinction between $H^1(\PP, \Sigma; \R)$ and $C^1(\PP,
\Sigma; \R)$ (the space of 1-cochains). 
}

We have shown that
there is a lift for any $\A$. It will be important for us that there
are many possible lifts; indeed in Masur's construction above, there
are many possible choices of $\B$ which would have worked just as well
as \equ{eq: Masur choice}. The set of all such $\B$'s will be
identified in
\S\ref{section: pairs}.


}

\subsection{Decomposition associated with a transverse system}
\name{subsection: decompositions}
Suppose $\FF$ is an oriented singular foliation on $(S, \Sigma)$ and
$\gamma: J \to S$ is a transverse system to 
$\FF$. There is an associated 
cellular decomposition $\BB=\BB(\gamma)$ of $(S, \Sigma)$ defined as
follows. Let $\IE = \IE(\FF, \gamma)$ be the generalized interval
exchange corresponding to $\gamma$. 

The 2-cells in $\BB$ correspond to the intervals of continuity 
of $\IE$. For each such interval $I$, the corresponding cell consists
of the union of interiors of leaf intervals beginning at $I$ and
ending at $\IE(I)$. Hence it fibers over $I$ and hence has the
structure of an open topological rectangle. The boundary of a 2-cell
lies in $\gamma$ and in certain segments of leaves, and the union of
these form the 1-skeleton. The 0-skeleton consists of points of
$\Sigma$, endpoints of $\gamma$, and points of discontinuity of $\IE$
and $\IE^{-1}$. Edges of the 1-skeleton lying on $\gamma$ will be
called {\em transverse edges} and edges lying on $\FF$ will be called {\em
  leaf edges}. Leaf edges inherit an orientation from $\FF$ and
transverge edges inherit the transverse orientation induced by $\FF$. 

Note that opposite boundaries of a 2-cell could come from the same
points in $S$: a
particular example occurs for a special transverse system, where if
there is a transverse edge crossing a cylinder, that cylinder is
obtained as a single 2-cell with its bottom and top edges
identified. Such a 2-cell is called a {\em cylinder cell}. 

It is helpful to consider a {\em spine} for $\BB$, which we denote
$\chi = \chi(\gamma)$, and is composed of the 1-skeleton of $\BB$
together with one leaf $\ell_R$ for every rectangle $R$, traversing
$R$ from bottom to top. The spine is closely related to Thurston's
train tracks; indeed, if we delete from $\chi$ the singular points
$\Sigma$ and the leaf edges that meet them, and collapse each element
of the transversal $\gamma$ to a point, we obtain a train track that
`carries $\FF$' in Thurston's sense. But note that keeping the deleted
edges allows us to keep track of information relative to $\Sigma$, and
in particular, to keep track of the saddle connections in $\FF$.

\combarak{This was not precise enough. I am guessing that $Q(\A, \B)$
is the intersection of $\A$ and $\B$ when $\A$ is thought of as an
element of $H_1(S, \Sigma)$ and $\B$, as an element of $H^1(S, \Sigma)
\cong H_1(S \sm \Sigma)$. Is this right?}

\combarak{A remark on the relation of $\BB(\gamma)$ to train tracks
belongs here.}

\subsection{Transverse cocycles, homology and
cohomology}\name{subsection: cohomology}
We now describe cycles supported on a foliation $\FF$ and their dual
cocycles. 

We will see that a transverse measure $\mu$ on $\FF$ defines an
element $[c_\mu] \in H_1(S, \Sigma)$, expressed concretely as a cycle
$c_\mu$ in the spine of $\chi(\gamma)$ of a transverse
system. Poincar\'e duality identifies $H_1(S, \Sigma)$ with $H^1(S \sm
\Sigma)$, and the dual $[d_\mu]$ of $[c_\mu]$ is represented by the
cochain corresponding to integrating the measure $\mu$. 

If $\mu$ has no atoms, then in fact we obtain $[c'_\mu] \in H_1(S \sm
\Sigma)$, and its dual $[d'_mu]$ lies in $H^1(S, \Sigma)$. The natural
maps $H_1(S \sm \Sigma) \to H_1(S, \Sigma)$ and $H^1(S, \Sigma) \to
H^1(S \sm \Sigma)$ take $[c'_\mu]$ to $[c_\mu]$ and $[d'_\mu]$ to
$[d_\mu]$ respectively. 

We will now describe these constructions in more detail.   

Let $\gamma$ be a transverse system and $\chi(\gamma)$ the spine of
its associated complex $\BB(\gamma)$ as above. Given $\mu$ we define a
1-chain on $\chi$ as follows. For each rectangle $R$ whose bottom side
is an interval $\kappa$ in $\gamma$, set $\mu(R) = \mu(\interior \,
\kappa)$ (using the interior is important here because of possible
atoms in the boundary). For each leaf edge $f$ of $\BB$, set 
$\mu(\{f\})$ to be the transverse measure of $\mu$ across $f$ (which
is 0 unless the leaf $f$ is an atom of $\mu$). The 1-chain 
$x = \sum_R \mu(R) \ell_R + \sum_f \mu(\{ f\}) f + z$ may not be a
cycle, but we note that invariance of $\mu$ implies that, on each
component of $\gamma$, teh sum of measures taken with sign (ingoing
vs. outgoing) is 0, so that $\partial x$ restricted to each component
is null-homologous. Hence by `coning off' $\partial x$ in each
component of $\gamma$ we can obtain a cycle of the form:

\eq{eq: explicit form cycle}{
c_\mu = \sum_{R \ \mathrm{rectangle}} \mu(R) \ell_R + \sum_{f
  \ \mathrm{leaf \ edge \ of \ } \BB} \mu(\{f\}) f + z, 
} 
where $z$ is a 1-chain supported in $\gamma$ such that $\partial z =
-\partial x$. Invariance and additivity of $\mu$ imply that the
homology class in $H_1(S, \Sigma)$ is independent of the choice of
$\gamma$. 

The cochain $d_\mu$ is constructed as follows: in any product
neighborhood $U$ for $\FF$ in $S \sm \Sigma$, integration of $\mu$
gives a map $U \to \R$, constant along leaves (but discontinuous at
atomic leaves). On any oriented path in $U$ with endpoints off the
atoms, the value of the cochain is obtained by mapping endpoints to
$\R$ and subtracting. Via subdivision this extends to to a cochain
defined on 1-chains whose boundary misses atomic leaves. This cochain
is a cocycle via additivity and invariance of the measures, and
suffices to give a cohomology class (or one may extend it to all
1-chains by a suitable chain-homotopy perturbing vertices on atomic
leaves slightly). 

In the case with no atoms, we note that the expression for $c_\mu$ has
no terms of the form $\mu(\{f\}) f$, and hence we get a cocycle in $S
\sm \Sigma$. The definition of the cochain extends in that case to
neighborhoods of singular points, and evaluates consistently on
relative 1-chains, giving a class in $H^1(S, \Sigma)$. 

A cycle corresponding to a transverse measure will be called a {\em
  (relative) cycle carried by $\FF$}. The set of all (relative) cycles
carried by $\FF$ is a convex cone in $H_1(S, \Sigma; \R)$ which we
denote by $H^\FF_+$. Since we allow atomic measures, we can think of
(positively oriented) saddle connections or closed leaves in $\FF$ as
elements of $H^\FF_+$. Another way of constructing cycles carried by
$\FF$ is the Schwartzman asymptotic cycle construction
\cite{Schwartzman}. 
These are projectivized limits of long loops which are mostly in $\FF$
but may be 
closed by short segments transverse to $\FF$. 
It is easy to see that $H^{\FF}_+ \cap H_1(S)$ is the convex cone
over the asymptotic cycles, or equivalently the image of the
non-atomic transverse measures, and that $H^{\FF}_+$ is the convex
cone over asymptotic cycles and positive saddle connections in $\FF$. 
Generically (when $\FF$ is uniquely ergodic and contains no
saddle connections) $\HH_+^\FF$ is one-dimensional
and is spanned by the so-called {\em asymptotic cycle} of
$\FF$.

\subsection{Intersection pairing} \name{subsection: intersection}

Via Poincar\'e duality, the canonical pairing on $H^1(S, \Sigma) \times H_1(S, \Sigma)$ becomes 
the intersection pairing on $H_1(S \sm \Sigma) \times H_1(S,
\Sigma)$. In the former case we denote this pairing by $(d, c) \mapsto
d(c)$, and in the latter, by $(c,c') \mapsto c \cdot c'$. Suppose
$\FF$ and $\GG$ are two mutually transverse oriented singular
foliations, with transverse measures $\mu$ and $\nu$ respectively. If
we allow $\mu$ but not $\nu$ to have atoms, then $[c_\mu] \in H_1(S,
\Sigma)$ and $[c_\nu] \in H_1(S \sm \Sigma)$ so we have the
intersection pairing 
\eq{eq: foliation pairing}{
c_\nu \cdot c_\mu = d_{\nu}(c_\mu) = \int_S \nu \times \mu.
}
In other words we integrate the transverse measure of $\GG$ along the
leaves of $\FF$, and then integrate against the transverse measure of
$\FF$. (The sign of $\nu \times \mu$ should be chosen so that it is
positive when the orientation of $\GG$ agrees with the transverse
orientation of $\FF$). We can see this explicitly by choosing a
transversal $\gamma$ for $\FF$ lying in the leaves of $\GG$ . Then the
cochain representing $d_\nu$ is 0 along $\gamma$, and using the form
\equ{eq: explicit form cycle} we have
\eq{eq: pair c d}{
d_\nu(c_\mu) = \sum_R \mu(R) \int_{\ell_R} \nu \, + \sum_f \mu(\{f\})
\int_f \nu.
}

\subsubsection{Judicious case} Now suppose $\gamma$ is a judicious
transversal. In this case the pairing of $H^1(S, \Sigma)$ and $H_1(S,
\Sigma)$ has a concrete form which we will use in \S \ref{section:
  pairs}. 

There is a cell decomposition  $\mathcal{D}$ of $S$ that is dual to
$\BB$, defined as follows. Because $\gamma$ intersects all leaves, and
terminates at $\Sigma$ on both ends, each rectangle $R$ of $\BB$ has
exactly one point of $\Sigma$ on each of its leaf edges. Connect these
two points by a transverse arc in $R$ and let $\mathcal{D}^1$ be the
union of these arcs. $\mathcal{D}^1$ cuts $S$ into a disk
$\mathcal{D}^2$, bisected by $\gamma$. Indeed, upward flow from
$\gamma$ encounters $\mathcal{D}^1$ in a sequence of edges which is
the {\em upper boundary} of the disk, and downward flow encounters the
{\em lower boundary} which goes through the edges of $\mathcal{D}^1$
in a permuted order, in fact exactly the permutation $\sigma$ of the
interval exchange $\IE(\FF, \gamma)$. 

A class in $H^1(S, \Sigma)$ is determined by its values on the
(oriented) edges of $\mathcal{D}^1$, and in fact this gives a basis,
which we can label by the intervals of continuity of $\IE(\FF,
\gamma)$ (the condition that the sum is 0 around the boundary of the
disk is satisfied automatically). The Poincar\'e dual basis for $H_1(S
\sm \Sigma)$ is given by the loops $\hat{\ell}_R$ obtained by joining
the endpoints of $\ell_R$ along $\gamma$. 

The pairing restricted to non-negative homology is computed by the
form $Q$ of \equ{eq: defn Q1}. Ordering the rectangles $R_1, \ldots,
R_d$ according to their bottom arcs along $\gamma$ and writing
$\hat{\ell}_i = \hat{\ell}_{R_i}$, we note that $\hat{\ell}_i$ and
$\hat{\ell}_j$ have nonzero intersection number precisely when the
order of $i$ and $j$ is reversed by $\sigma$
(i.e. $(i-j)(\sigma(i)-\sigma(j)) < 0$), and in particular (accounting
for sign), 
\eq{eq:Q intersection}{
\hat{\ell}_i \cdot \hat{\ell}_j = Q(\E_i, \E_j).
}
In other words, $Q$ is the intersection pairing on $H_1(S \sm \Sigma)
\times H_1(S \sm \Sigma)$ with the given basis (note that this form is
degenerate, as the map $H_1(S \sm \Sigma) \to H_1(S, \Sigma)$ has a
kernel).

\section{The lifting problem}
In this section we prove Theorem \ref{thm: sullivan1}. 


\begin{proof}[Proof of Theorem \ref{thm: sullivan1}]
We first explain the easy direction (1) $\Longrightarrow$ (2). Since 
$\FF$ is everywhere transverse to $\GG$, and the surface is connected,
reversing the orientation of $\GG$ if necessary we can assume
that positively oriented paths in leaves of $\FF$ always cross $\GG$
from left to right. Therefore, using \equ{eq: foliation pairing} and
\equ{eq: pair c d}, we find that 
$\B(\delta) > 0$ for any $\delta \in \HH_+^\FF$. 

Before proving the converse we indicate the idea of proof. We will
consider a sequence of finer and finer cell decompositions associated
to a shrinking sequence of special transverse systems. As the
transversals shrink, the associated train tracks split, each one being
carried by the previous one. We examine the weight that a
representative of $\B$ places on the vertical leaves in the cells of
these decompositions (roughly speaking the branches of the associated
train tracks). If any of these remain non-positive for all time, then
a limiting argument produces an invariant measure on $\FF$ which has
non-positive pairing with $\B$, a contradiction. Hence eventually all
cells have positive `heights' with respect to $\B$, and can be
geometrically realized as rectangles. The proof is made complicated by
the need to keep track of the singularities, and in particular by the
appearance of cylinder cells in the decomposition.

Fix a 
special transverse system  $\gamma$ for $\FF$ (see \S \ref{subsection:
transversal}), and let $\BB = \BB(\gamma)$ be the corresponding cell
decomposition as in \S\ref{subsection: decompositions}. Given a path
$\alpha$ on $S$ which is contained in a 
leaf of $\FF$ and begins and ends in transverse edges of $\BB$, we
will say that $\alpha$ 
is {\em parallel to saddle connections} if there is a continuous
family of arcs $\alpha_s$ contained in leaves, where $\alpha_0 =
\alpha$, $\alpha_1$ is a union of saddle connections in $\FF$, and the
endpoints of each $\alpha_s$ are in $\gamma$. 

We claim 
that for any $N$ there is a special transverse system
$\gamma_N \subset \gamma$ such that the following holds: for any
leaf edge $e$ of $\BB_N = \BB(\gamma_N)$, either $e$ is parallel to
saddle connections or, when moving along $e$ from bottom to top, we
return to $\gamma \sm \gamma_N$ at least $N$ times before returning to
$\gamma_N$.  
Indeed, if the claim is false then for some $N$ and any special
$\gamma' \subset \gamma$ there is a leaf 
edge $e'$ in $\BB_N$ which is not parallel to saddle connections, 
and starts and ends at points $x$, $y$ in $\gamma'$, making at
most $N$ crossings with $\gamma \sm \gamma'$. 
Now take $\gamma'_j$ to be shorter
and shorter special transverse subsystems of $\gamma$, denote
the corresponding edge by $e_j$ and the points by $x_j, \,
y_j$. Passing to
a subsequence, we find that $x_j,y_j$ converge to points in $\Sigma$
and $e_j$ converges to a concatentation of at most $N$ saddle
connections joining these points. In particular, for large enough $j$,
$e_j$ is parallel to saddle connections, a contradiction proving the claim. 

Since $\gamma$ is special, each periodic component of $\FF$ consists
of one 2-cell of $\BB$, called a {\em cylinder cell}, with its top and
bottom boundaries identified along an edge traversing the component,
which we call a {\em cylinder transverse edge}. The same holds for
$\BB_N$, and our construction ensures that $\BB$ and $\BB_N$ contain
the same cylinder cells and the same cylinder transverse edges. 

Let $\beta$ be a singular (relative) 1-cocycle in $(S, \Sigma)$
representing $\B$. We claim that we may choose $\beta$ so that it
vanishes on non-cylinder transverse edges. Indeed, each such edge
meets $\Sigma$ only at one endpoint, so they can be
deformation-retracted to $\Sigma$, and pulling back a cocycle via this
retraction gives $\beta$. Note that $\beta$ assigns a well-defined
value to the periodic leaf edges, namely the value of $\B$ on the
corresponding loops. 

In general $\beta$ may assign non-positive heights to rectangles, and
thus it does not assign any reasonable geometry to $\BB$. We now claim
that there is a positive $N$ such that for any leaf edge $e \in
\BB_N$, 
$$\beta(e)>0.$$
Since a saddle connection in $\FF$ represents an element of
$H^{\FF}_+$, our assumption implies that 
$\beta(e)>0$ for any leaf edge of $\BB_N$ which is a saddle
connection. Moreover by construction, if $e$ is parallel to saddle
connections, then $\beta(e) = \sum \beta(e_i)$ for
saddle connections $e_i$, so again $\beta(e)>0$. Now suppose by contradiction
that for any $N$ we can find a leaf edge $e_N$ in $\BB_N$ which is not
parallel to saddle connections and such that
$C_N = \left|e_N \cap \gamma  \right| \geq N$
and 
$$\limsup_{N \to \infty} \frac{\beta(e_N)}{C_N} \leq 0.$$
Passing to a subsequence we define
a measure $\mu$
on $\gamma$ as a weak-* limit of the measures 
$$\nu_N(I) = \frac{|e_N \cap I|}{C_N}, \ \
\mathrm{where \ } I \subset \gamma \mathrm{ \ is \ an \ interval;}$$
it is invariant under the
return map to $\gamma$ and thus defines a transverse measure on $\FF$
representing a class $[c_\mu] \in H_+^\FF$. Moreover by construction it
has no atoms and gives measure zero to the cylinder cells. We will evaluate
$\beta(c_\mu)$.

For each rectangle $R$ in $\BB$, let $\theta_R \subset
\gamma$ be the transverse arc on the bottom of $R$ and $\ell_R$ a leaf
segment going through $R$ from bottom to top, as in \S
\ref{subsection: decompositions}. 
Since $e_N$ is not parallel to saddle connections, its intersection
with each $R$ is a union of arcs parallel to $\ell_R$. Since $\beta$
gives all such arcs the same values $\beta(\ell_R)$, we have
$$
\beta(e_N)=\beta\left( 
\sum_R |e_N \cap \theta_R| \ell_R
\right) = \sum_R |e_N \cap \theta_R| \beta (\ell_R).
$$
By \equ{eq: explicit form cycle}, since $\mu$ has no atoms we can
write 
$$
c_\mu = \sum_R \mu(\theta_R) \ell_R +z
$$
where $z \subset \gamma$. Since $\beta$ vanishes along $\gamma$ we
have
$$
\B(c_\mu) =  \sum_R \mu(\theta_R) \beta(\ell_R) = \sum_R \lim_N
\frac{|e_N \cap \theta_R}{C_N} \beta(\ell_R) = \lim_N
\frac{\beta(e_N)}{C_N} \leq 0.
$$
This contradicts the hypothesis, proving the claim. 

The claim implies that the
topological rectangles in $\BB_N$ can be given a compatible Euclidean
structure, using the transverse measure of $\FF$ and $\beta$ to measure respectively
the horizontal and vertical components of all relevant edges. Note
that all non-cylinder cells become metric rectangles, and the cylinder
cells become metric parallelograms. Thus we have constructed a
translation surface structure on $(S, \Sigma)$ whose horizontal foliation $\GG$
represents $\beta$, as required. 
\end{proof}

\section{The homeomorphism theorem}\name{section: dev homeo}
We now prove Theorem \ref{thm: hol homeo}, which states that 
$$
\hol: \til \HH(\FF) \to \AAA(\FF) \times \BBB(\FF)
$$
is a homeomorphism, where $\til \HH(\FF)$ is the set of marked
translation surface structures with vertical foliation topologically
equivalent to $\FF$, $\AAA(\FF) \subset H^1(S, \Sigma)$ is the
set of Poincar\'e duals of asymptotic cycles of $\FF$, and $\BBB(\FF)
\subset H^1(S, \Sigma)$ is the set of $\B$ such that $\B(\alpha)>0$
for all $\alpha \in H^{\FF}_+$. 

\begin{proof}
The fact that $\hol$ maps  $\til \HH(\FF)$ to $\AAA(\FF) \times
\BBB(\FF)$ is an immediate consequence of the definitions, and of the
easy direction of Theorem \ref{thm: sullivan1}. That it is continuous
is also clear from definitions. That $\hol|_{\til \HH(\FF)}$ maps onto
$\AAA(\FF) \times \BBB(\FF)$ is the hard direction of Theorem
\ref{thm: sullivan1}. Injectivity is a consequence of the following:

\begin{lem}\name{lem: fiber singleton}
Let $\til \HH$ be a stratum of marked translation surfaces of type
$(S, \Sigma)$. Fix a singular measured foliation on $(S, \Sigma)$, and
let $\B \in H^1(S, \Sigma; \R)$. Then there is at most one $\bq \in
\til \HH$ with vertical foliation $\FF$ and horizontal foliation
representing $\B$.  
\end{lem}
\begin{proof}[Proof of Lemma \ref{lem: fiber singleton}]
Suppose that $\bq_1$ and $\bq_2$ are two marked translation surfaces,
such that the vertical measured foliation of both is $\FF$, and the
horizontal measured foliations $\GG_1, \GG_2$ both represent $\B$. We
need to show that $\bq_1=\bq_2$. Let $\gamma$ be a special transverse
system to $\FF$ (as in \S \ref{subsection: transversal} and the proof of Theorem \ref{thm:
  sullivan1}). Recall that the non-cylinder edges of $\gamma$ can be
made as small as we like. Since $\FF$ and $\GG_1$ are transverse, we
may take each non-cylinder segment of $\gamma$ to be contained in
leaves of $\GG_1$. 

In a sufficiently small neighborhood $U$ of any $p \in \Sigma$, we may
perform an isotopy of $\bq_2$, preserving the leaves of $\FF$, so as
to make $\GG_2$ coincide with $\GG_1$. This follows from the fact that
in $\R^2$, the leaves of any foliation transverse to the vertical
foliation can be expressed as graphs over the horizontal
direction. Having done this, we may choose the non-cylinder segments
of $\gamma$ to be contained in such neighborhoods, and hence
simultaneously in leaves of $\GG_1$ and $\GG_2$. 

\ignore{

it has one cylinder edge going across
each cylinder of periodic leaves for $\FF$, and its other segments
start at singularities and go into minimal components of $\FF$ without
crossing saddle connections in $\FF$. Make $\gamma$ smaller by
deleting segments, so it has only one segment going into each minimal
component. Since $\FF$ and $\GG_1$ are
transverse, we may take $\gamma$ to be contained in leaves of
$\GG_1$. 

We claim that, after making $\gamma$ smaller if necessary, we may
precompose $\bq_2$ with an isotopy rel 
$\Sigma$, in order to assume that $\gamma$ is also contained in leaves
of $\GG_2$, i.e. that $\GG_1$ and $\GG_2$ coincide along
$\gamma$. Indeed, if $\alpha$ is a sufficiently short arc in $\gamma$
going from a singularity $x$ into a minimal set $M$, we may isotope
along $\FF$ in $M$ 
until the leaf of $\GG_2$ starting at $x$ reaches $\alpha$. Since $M$
contains no other segments of $\gamma$ we may perform this operation
independently on each minimal component. Now suppose $\alpha$ is a
cylinder edge in $\gamma$, so that $\alpha$ connects two
singularities on opposite sides of a cylinder $P$. In particular
$\int_\alpha \GG_1 = \B(\alpha) = \int_\alpha \GG_2$, so we may
perform an isotopy inside $P$ until $\alpha$ is also contained in a
leaf of $\GG_2$, as claimed. 
}

Now consider the cell decomposition $\BB=\BB(\gamma)$, and let $\beta_i =
[\GG_i]$ be the 1-cocycle on $\BB$ obtained by integrating
$\GG_i$. For a transverse non-cylinder edge $e$ we have $\beta_1(e)=\beta_2(e)=0$
since $e$ is a leaf for both foliations. If $e$ is contained in a leaf
of $\FF$, we may join its endpoints to $\Sigma$ by paths $d$ and $f$
along $\gamma$. Then $\delta = d+ e+ f$ represents an element of 
$H_1(S, \Sigma)$ so that $\beta_1(\delta) =
\beta_2(\delta)=\BB(\delta)$. Since $\beta_i(d)= \beta_i(f)=0$ we have
$\beta_i(e) = \BB(\delta)$. For a cylinder edge $e$, its endpoints are
already on $\Sigma$ so $\beta_1(e)=\beta_2(e)$. We have shown that
$\beta_1 = \beta_2$ on all edges of $\BB$. 

Recall from the proof of Theorem \ref{thm: sullivan1} that $\bq_i$ may
be obtained explicitly by giving each cell the structure of a
Euclidean rectangle or parallelogram (the latter for cylinder cells)
as determined by $\FF$ and $\beta_i$ on the edges. Therefore $\bq_1 =
\bq_2$. 
\end{proof}

Finally we need to show that the inverse of hol is continuous. This is
an elaboration of the well-known fact that hol is a local
homeomorphism, which we can see as follows. Let $\bq \in \til \HH$,
and consider a geometric triangulation $\tau$ of $\bq$ with vertices
in $\Sigma$ (e.g. a Delaunay triangulation \cite{MS}). The shape of
each triangle is uniquely and continuously determined by the hol image
of each of its edges. Hence if we choose a neighborhood $\mathcal{U}$ of $\bq$
small enough so that none of the triangles becomes degenerate, we have
a homeomorphism $\hol: \mathcal{U} \to \mathcal{V}$ where $\mathcal{V}
= \hol(\mathcal{U}) \subset H^1(S, 
\Sigma; \R^2)$. 

If for $\bq' \in \mathcal{U}$, the first coordinate $x (\bq')$ of $\hol(\bq')$
lies in $\AAA(\FF)$, then we claim that $\bq' \in \til \HH(\FF)$. This
is because the vertical foliation is determined by the weights that
$x(\bq)$ assigns to edges of the triangulation. By Lemma \ref{lem:
  fiber singleton},
$\bq'$ is the unique preimage of $\hol(\bq')$ in $\til
\HH(\FF)$. Hence $\left(\hol|_{\mathcal{U}}\right)^{-1}$ and $\left(\hol|_{\til \HH(\FF)}\right)^{-1}$
coincide on their overlap, so continuity of one implies continuity of
the other.
\end{proof}

\section{Positive pairs}\name{section: pairs}
We now reformulate Theorem \ref{thm: sullivan1} in the language of
interval exchanges, and derive several useful consequences. For this
we need some more definitions. 
Let the notation be as in \S\ref{subsection: iets}, so that $\sigma$
is an irreducible and 
admissible permutation on $d$ elements. 
The tangent space $T\R^d_+$ has a natural product structure $T\R^d_+=\R^d_+
\times \R^d$ and a corresponding affine structure. 
Given $\A \in \R^d_+, \, \B \in \R^d$, we can think of $(\A, \B)$ as
an element of $T\R^d_+$. We will be using the same symbols $\A,
\B$ which were previously used to denote cohomology classes; the
reason for this will become clear momentarily. Let $\IE =
\IE_{\sigma}(\A)$ be the interval exchange associated with $\sigma$
and $\A$. 

For $\B \in
\R^d$, in analogy with \equ{eq: disc points2}, 
define
$y_i(\B), \, y'_i(\B)$ via \eq{eq: disc points1}
{
y_i=y_i(\B)= \sum_{j=1}^i b_j, \ \ \
y'_i =y'_i(\B) = \sum_{j=1}^i b_{\sigma^{-1}(j)} =
\sum_{\sigma(k)\leq i} b_k.
}
In the case of Masur's construction (Figure \ref{figure: Masur}), the $y_i$ are the
heights of the points in the upper boundary of the polygon, and the
$y'_j$ in the lower. 

Consider the following step functions $f,g, L: I \to \R$, depending on
$\A$ and $\B$:
\eq{eq: defn L}{
\begin{split}
f(x) & = y_i \ \ \ \ \ \ \ \ \mathrm{for}\ x \in I_i = [x_{i-1}, x_i) \\
g(x) & = y'_i \ \ \ \ \ \ \ \ \mathrm{for} \ x \in I'_i = [x'_{i-1}, x'_i) \\
L(x) & = f(x) - g(\IE(x)).
\end{split}
}
Note that for $Q$ as in \equ{eq: defn Q1} and $x \in I_i$ we have 
\eq{eq: relation L Q}{
L(x) = Q(\E_i, \B) .
} 

If there are $i,j \in \{0, \ldots , d-1\}$
 (not necessarily distinct)
and $m>0$ such that $\IE^m(x_i)=x_j$ we will say that $(i,j,m)$ is a {\em 
connection} for $\IE$. 
We denote the set of invariant non-atomic
probability measures for $\IE$ by $\MM_\A$, and 
the set of connections by $\LL_\A$.

\begin{Def}
We say that $(\A, \B) \in \R^d_+ \times \R^d$ is {\em a positive pair} if 
\eq{eq: positivity}{\int L \, d\mu >0
 \
\ \mathrm{for \ any \ } \mu \in \MM_{\A}
}
and 
\eq{eq: connections positive}{\sum_{n=0}^{m-1}L(\IE^nx_i) > y_i-y_j \ \
\ \mathrm{for \ any \ } (i,j,m) \in \LL_{\A}.
}
\end{Def}
As explained in \S \ref{subsection: transversal}, following \cite{ZK}
one can construct a surface $S$ with a finite subset $\Sigma$,
a foliation $\FF$ on $(S, \Sigma)$, and a judicious curve $\gamma$ on
$S$ such that $\IE_\sigma(\A) = \IE(\FF, \gamma)$ (we identify
$\gamma$ with $I$ via the transverse measure). 

Moreover, as in \S \ref{subsection: intersection} there is a complex
$\mathcal{D}$ with a single 2-cell $\mathcal{D}^2$ containing $\gamma$
as a properly embedded arc, and whose boundary is divided by $\partial
\, \gamma$ into two arcs each of surjects to the 1-skeleton of
$\mathcal{D}$. The upper arc is divided by $\Sigma$ into $d$ oriented
segments $K_1, \ldots, K_d$ images of the segments $I_i$ under flow
along $\FF$. The vector $\B$ can be interpreted as a class in $H^1(S,
\Sigma)$ by assigning $b_i$ to the segment $K_i$. The Poincar\'e dual
of $\B$ in $H_1(S \sm \Sigma)$ is written $\beta = \sum b_i \hat{\ell}_i$,
where $\hat{\ell}_i$ are as in \S \ref{subsection: intersection}.  

Using these we show:
\begin{prop}
\name{prop: positive interpretation}
$(\A, \B)$ is  a positive pair if and only if $\B (\alpha) >0 $
for any $\alpha \in H^{\FF}_+$. 
\end{prop}
\begin{proof}
We show that implication $\implies$, the converse being similar. 
It suffices to consider the cases where 
$\alpha \in H^\FF_+$ corresponds to $\mu \in \MM_\IE$ or to a
positively oriented saddle connection in $\FF$, because a general
element in $H^\FF_+$ is a convex combination of these. 

In the first case, define $a'_k = \mu(I_k)$ and $\A' = \sum a'_k
\E_k$. The corresponding homology class in $H_1(S \sm \Sigma)$ is
$\alpha = \sum a'_k \hat{\ell}_k$. Hence, as we saw in \equ{eq:Q
  intersection},
$$
\B(\alpha) = \beta \cdot \alpha = Q(\A', \B)  = \sum a'_k Q(\E_k, \B)
\stackrel{\equ{eq: relation L Q}}{=} \sum \mu(I_k) L|_{I_k} = \int L
\, d\mu >0.  
$$ 

Now consider the second case. Given a connection $(i,j,m)$ for $\IE$,
the corresponding saddle connection $\alpha$ meets the disk
$\mathcal{D}^2$ in a union of leaf segments $\eta_1, \ldots, \eta_m$
where each $\eta_n$ is the leaf segment in $\mathcal{D}^2$
intersecting the interval $\gamma = I$ in the point $\IE^n(x_i)$. This
point lies in some interval $I_r$ and some interval $I'_s$. If we let
$\hat{\eta}_n$ be a line segment connecting the left endpoint of
$I'_s$ to the left endpoint of $I_r$ then the chain $\sum
\hat{\eta}_n$ is homologous to $\sum \eta_n$ (note that the first
endpoint of $\eta_1$ and the last endpoint of $\eta_m$ do not
change). Now we apply our cocycle $\B$ to each $\hat{\eta}_n$ to
obtain
$$
\B(\hat{\eta}_n) = y_r - y'_s = f(\IE^n(x_i)) - g(\IE^n(x_i)).
$$
Summing, we get
\[
\begin{split}
\B(\alpha) & = \B \left (\sum \eta_n \right)  = \B \left (\sum
\hat{\eta}_n \right)\\ & = \sum f(\IE^n(x_i)) - 
g(\IE^n(x_i)) \\
& = -f(x_i) + f(\IE^m(x_i)) + \sum_{n=0}^{m-1} f(\IE^n(x_i)) -
g(\IE^{n+1}(x_i)) \\ 
& = -y_i+y_j + \sum_{n=0}^{m-1} L(\IE^n(x_i)) >0.
\end{split}
\]
\end{proof}

Now we can state the interval exchange version of Theorem \ref{thm: sullivan1}:
\begin{thm}
\name{thm: sullivan2}
Let $\til \HH$ be the stratum of marked translation surfaces corresponding to
$\sigma$. Then for any positive pair $(\A_0, \B_0)$ there is a
neighborhood $\mathcal{U}$ of $(\A_0, \B_0)$ in $T\R^d_+$, and a map
$\bq: \mathcal{U} \to \til \HH$ such that the following hold:
\begin{itemize}
\item[(i)]
$\bq$ is an affine map and a local homeomorphism. 
\item[(ii)]
For any $(\A, \B) \in \mathcal{U}$, $\bq(\A,\B)$ is a lift of $\A$. 
\item[(iii)]
Any $(\A, \B)$ in $\mathcal{U}$ is positive.
\item[(iv)]
Suppose $(\A, \B) \in \mathcal{U}$ and $\vre_0>0$ is small enough so
that $(\A+s\B, \B) \in \mathcal{U}$ for $|s| \leq \vre_0$. Then $h_s \bq(\A, \B)
= \bq(\A+s\B, \B)$ for $|s | \leq \vre_0$.

\end{itemize}

\end{thm}

\begin{proof}
Above and in \S \ref{subsection: intersection} we identified $\R^d$
with $H^1(S, \Sigma; \R)$, obtaining an injective affine map  
$$ T\R^d_+ \cong \R^d_+ \times \R^d \to H^1(S, \Sigma;
\R)^2 \cong H^1(S, \Sigma; \R^2)$$
and we henceforth identify $T\R^d_+$ with its image. 

Given a positive pair $(\A_0, \B_0)$, as discussed at the end of \S
\ref{subsection: transversal} we obtain a surface $(S, \Sigma)$
together with a measured foliation $\FF = \FF(\A_0)$ and a judicious
transversal $\gamma_0$, so that the return map $\IE(\FF, \gamma)$ is
equal to $\IE_\sigma(\A_0)$, where $I$ parametrizes $\gamma_0$ via the
tranverse measure. 

The positivity condition, together with Proposition \ref{prop:
  positive interpretation} and Theorem
\ref{thm: sullivan1} give us a translation surface structure $\bq =
\bq(\A_0, \B_0)$ whose vertical foliation is $\FF$ and whose image
under hol is $(\A_0, \B_0)$. 

Let $\tau$ be a geometric triangulation on $\bq$. Assume for the
moment that $\tau$ has no vertical edges, and in particular that each
edge for $\tau$ is transverse to $\FF$. Now consider $\A$ very close
to $\A_0$. The map $\IE_\sigma(\A)$ is close to $\IE_\sigma(\A_0)$ and
hence induces a foliation $\FF(\A)$ whose leaves are nearly parallel
to those of $\FF(\A_0)$. More explicitly, $\FF(\A)$ is obtained by
modifying $\FF(\A_0)$ slightly in a small neighborhood of $\gamma$ so
that it remains transverse to both $\gamma$ and $\tau$, and so that
the return map becomes $\IE_\sigma(\A)$. This can be done if $\A$ is
in a sufficiently small neighborhood of $\A_0$. 

Now if $(\A, \B)$ is sufficiently close to $(\A_0, \B_0)$, the pair
$(\A, \B)$ assign to edges of $\tau$ vectors which retain the
orientation induced by $(\A_0, \B_0)$. Hence we obtain a new geometric
triangulation, for which $\FF(\A)$ is transverse to the edges, has
transverse measure agreeing with $\A$ on the edges, and hence is still
realized by the vertical foliation. Moreover $\gamma_0$ is still
transverse to the new foliation and the return map is the correct
one. That is what we wanted to show. 

Returning to the case where $\tau$ is allowed to have vertical edges:
note that at most one edge in a triangle can be vertical. Hence, if we
remove the vertical edges we are left with a decomposition whose cells
are Euclidean triangles and quadrilaterals, and whose edges are
tranverse to $\FF$. The above argument applies equally well to this
decomposition. 

\medskip

We have shown that in a neighborhood of $(\A_0, \B_0)$, the map $\bq$
maps to $\til \HH_\tau$, and is a local inverse for hol. Hence it is
affine and a local homeomorphism, establishing (i). Part (ii) is by
definition part of our construction. Part (iii) follows from the
implication (1) $\implies$ (2) in Theorem \ref{thm: sullivan1}. In
verifying (iv) we use \equ{eq: G action}. 
\end{proof}

The following useful observation follows immediately: 
\begin{cor}
\name{cor: realizable open}
The set of positive pairs is open.

\end{cor}

We also record the following corollary of our construction, coming
immediately from the description we gave of the variation of the
triangles in the geometric triangulation $\tau$: 

\begin{cor}\name{cor: same gamma}
Suppose $\sigma, \til \HH$, a positive pair $(\A_0, \B_0)$, and $\bq:
\mathcal{U} \to \til \HH$ are as in Theorem \ref{thm: sullivan2}. The
structures $\bq(\A, \B)$ can be chosen in their isotopy class so that
the following holds:
A single curve $\gamma$ is a judicious transversal for the vertical
foliation of $\bq(\A, \B)$ for all $(\A, \B) \in \mathcal{U}$; the
flat structures vary continuously with $(\A, \B)$, meaning that the
charts in the atlas, modulo translation, vary continuously; the return
map to $\gamma$ satisfies $\IE(\bq, \gamma) = \IE_\sigma(\A)$. In
particular, for any $\bq' = \bq(\A, \B)$, there is $\vre>0$ such that
for $|s|< \vre, \A(a) = \A+s\B$, we have 
$$
\IE_\sigma(\A(s)) = \IE(h_s \bq', \gamma).
$$
\end{cor}

\ignore{
\begin{remark}
Let $Q$ be as in \equ{eq: defn Q1}. Note that $L(x) = Q(\E_i, \B)$ 
for $x \in I_i$, and that if $\lambda$ is Lebesgue measure on $I$, then 
\eq{eq: defn Q}{
Q(\A, \B) = \sum a_i Q(\E_i, \B) = \sum a_i L|_{I_i} 
 =  \int L \, d\lambda
}
is the area of the polygon $\mathcal P$ defined in \equ{eq: defn
mathcalP}. Thus requirement \equ{eq: positivity} can be interpreted as
saying that the area of a surface with horizontal (resp. vertical)
measured foliation determined by $\A$ (resp. $\B$) should be positive
with respect to any transverse measure for $\A$. 

Note also that if we fix $\FF$ and $\gamma$ then $\R^d$ is identified
with $H_1(S \sm \Sigma; \R)$ via \equ{eq:
concrete cycle}. Under this identification, the standard basis vectors
$\EE_i$ cycles whose images in $H_1(S ; \R)$ is non-trivial, and examining
\equ{eq: defn Q1} one recognizes the intersection pairing on $H_1(S ;
\R)$. Thus \equ{eq: positivity} can also be interepreted as saying
that the classes represented by $\A, \B$ should have positive
intersection. However, \equ{eq: connections positive} admits no 
interpretation on $S$; it is

\end{remark}
}

\section{Mahler's question for interval exchanges and its generalizations}
\name{section: results on iets}
A vector $\xx \in \R^d$ is  {\em very well approximable} if
for some $\vre>0$ there are infinitely many $\pp \in \Z^d, q
\in \N$ satisfying $\|q\xx - \pp \| < q^{-(1/d + \vre)}.$ It is
a classical fact that almost every (with respect to Lebesgue measure) $\xx
\in \R^d$ is not very well approximable, but that the set of very
well approximable vectors is large in the sense of Hausdorff
dimension. 
Mahler conjectured in the 1930's that for almost every (with respect to
Lebesgue measure on the real line) $x \in \R$, the vector $\A(x)$ as in 
\equ{eq: mahler curve}, 
is not very well approximable. This famous conjecture was settled
by Sprindzhuk in the 1960's and spawned many additional questions of a
similar nature. A general formulation of the  
problem is to describe measures $\mu$ on $\R^d$ for which almost
every $\xx$ is not very well approximable. See \cite{dimasurvey} for
a survey. 

In this section we apply Theorems \ref{thm: sullivan1} and \ref{thm:
sullivan2} to analogous problems concerning 
interval exchange transformations. 
Fix a permutation $\sigma$ on $d$ symbols which is irreducible and
admissible. In answer to a conjecture of Keane, it was proved by
Masur \cite{Masur-Keane} and Veech \cite{Veech-zippers} that almost
every $\A$ (with respect to Lebesgue measure on $\R^d_+$) is uniquely
ergodic. On the other hand Masur and Smillie
\cite{MS} showed that the set of non-uniquely ergodic interval exchanges
is large in the sense of Hausdorff dimension. 
In this paper we consider the problem of describing  
measures $\mu$ on $\R^d_+$ such that $\mu$-a.e. $\A$ is uniquely
ergodic. In a celebrated paper \cite{KMS}, it was shown that for
certain $\sigma$ and certain line 
segments $\ell \subset \R^d_+$ (arising from a problem in billiards on
polygons), for $\mu$-almost every $\A \in \ell$,
 $\IE_\sigma(\A)$ is uniquely
ergodic, where $\mu$ denotes Lebesgue measure on $\ell$. This was
later abstracted in 
\cite{Veech-fractals}, where the same result was shown to hold for a
general class of measures in place of $\mu$. Our strategy is strongly
influenced by these papers. 

Before stating our results we introduce more terminology. Let $B(x,r)$
denote the interval $(x-r, x+r)$ in $\R$.  
We say that a finite regular Borel measure $\mu$ on $\R$ is {\em
decaying and Federer} if there are positive  
$C, \alpha, D$ such that 
for every $x \in
\supp \, \mu$ and every $0<
\vre, r<1$, 
\eq{eq: decaying and federer}{
\mu \left(B(x,\vre r) \right) \leq C\vre^\alpha \mu\left(B(x,r) \right)
 \ \ \ \mathrm{and} \ \ \ \mu \left( B(x, 3r) \right) \leq D\mu\left(B(x,r) \right).
}
It is not hard to show that Lebesgue measure, and the
coin-tossing measure on Cantor's middle thirds set, are both decaying
and Federer. More constructions of such measures are given in
\cite{Veech-fractals, bad}. 
Let $\dim$ denote Hausdorff dimension, and for $x\in\supp\,\mu$ let
$$
\underline{d}_\mu(x) = \liminf_{r\to 0}\frac{\log 
\mu\big(B(x,r)\big)}{\log r}. 
$$

Now let
\eq{eq: defn epsn}{
\vre_n(\A) = \min \left \{\left|\IE^k(x_i) - \IE^{\ell}(x_j)\right| :
0\leq k, \ell \leq n, 1 \leq i ,j \leq d-1, (i,k) \neq (j,\ell) 
\right\},
}
where $\IE = \IE_\sigma(\A)$. 
We say that $\A$ is
{\em of recurrence type} if $\limsup n\vre_n(\A)
>0$
and 
{\em of bounded type} if $\liminf
n\vre_n(\A)>0$. 
It is known by work of Masur, Boshernitzan, Veech and Cheung that
if $\A$ is of recurrence type then it is uniquely ergodic, but that
the converse does not hold -- see \S \ref{section: saddles} below for more details. 

We have:
\begin{thm}[Lines]
\name{cor: inheritance, lines}
Suppose $(\A, \B)$ is a positive pair. Then there is $\vre_0>0$ such
that the following hold for $\A(s) = \A+s\B$ and 
for every decaying and Federer measure $\mu$ with $\supp \, \mu
\subset (-\vre_0, \vre_0)$:
\begin{itemize}
\item[(a)]
For $\mu$-almost every $s$, $\A(s)$ is of
recurrence type.   

\item[(b)]
$\dim \, \left \{s \in \supp \, \mu : \A(s) \mathrm{ \ is \ of \ bounded \
type} \right \}  \geq \inf_{x \in \supp \, \mu} \underline{d}_{\mu}(x).$

\item[(c)]
$\dim \left \{s \in (-\vre_0, \vre_0) : \A(s) \mathrm{ \ is \ not \ of \
recurrence \ type } \right \} \leq 1/2.$
\end{itemize}
\end{thm}

\begin{thm}[Curves]
\name{cor: ue on curves}
Let $I$ be an interval, let $\mu$ be a decaying and Federer measure on
$I$, and let $\beta: I \to
\R^d_+$ be a $C^2$ curve, such that for $\mu$-a.e. $s \in I$,
$(\beta(s), \beta'(s))$ is positive. Then for $\mu$-a.e. $s \in I$, 
$\beta(s)$ is of recurrence type. 
\end{thm}

Following some preliminary work, we will prove Theorem \ref{cor:
  inheritance, lines} in \S
\ref{section: inheritance lines} and Theorem \ref{cor: ue on curves} in \S
\ref{section: mahler curves}. In \S \ref{section: nondivergence} we
will prove a strengthening of Theorem \ref{cor: inheritance,
  lines}(a). 


\section{Saddle connections, compactness criteria}
\name{section: saddles}
A link between the $G$-action and unique ergodicity questions was made
in the following fundamental result.
\begin{lem}[{Masur \cite{Masur Duke}}]\name{lem: Masur}
If $q \in \HH$ is not uniquely ergodic then the trajectory
$\{g_tq: t \geq 0\}$ is {\em divergent}, i.e. for any compact
$K \subset \HH$ there is $t_0$ such that for all $t
\geq t_0$, $g_tq \notin K.$ 
\end{lem}
Masur's result is in fact stronger as it provides divergence in
the moduli space of quadratic differentials. The converse statement is
not true, see \cite{CM}. It is known (see \cite[Prop. 3.12]{Vorobets}
and \cite[\S 2]{Veech-fractals}) 
that $\A$ is of recurrence (resp. bounded) type 
if and only if the forward geodesic trajectory of any of its lifts
returns infinitely often to (resp. stays in) some compact subset of
$\HH$. It follows using Lemma \ref{lem: Masur} that if $\IE$ is 
of recurrence type then it is uniquely ergodic. In this section we
will prove a quantitative version of these results, linking the
behavior of $G$-orbits to the size of the quantity $n \vre_n(\A)$. 

We denote the set
of all saddle connections for a marked translation surface $\bq$ by
$\LL_\bq$. There is a natural 
identification of $\LL_\bq$ with $\LL_{g\bq}$ for any $g \in G$. We define 
$$\phi(q) = \min \left\{ \ell(\alpha,\bq ) : \alpha \in \LL_\bq
\right\},$$
where $\bq \in \pi^{-1}(q)$ and 
$\ell(\alpha, \bq) = \max \{|x(\alpha, \bq) |,
|y(\alpha, \bq) | \}.$ Let 
$\HH_1$ be the area-one sublocus in $\HH$, i.e. the set of $q \in \HH$
for which the total area of the surface is one. 
A standard compactness criterion for each stratum asserts that a set
$X\subset \HH_1$ is compact if and only if 
$$\inf_{q \in X} \phi(q) >0.$$
Thus, for each $\vre>0$, 
$$K_{\vre} = \left\{q \in \HH_1: \phi(q) \geq \vre \right\}
$$
is compact, and $\{K_\vre\}_{\vre >0}$ form an exhaustion of
$\HH_1$. We have:

\begin{prop}
\name{lem: relating lengths and disconts}
Suppose $\gamma$ is judicious for $q$. Then there
are positive $\kappa, c_1, c_2, n_0$ such that for $\IE = \IE(q,
\gamma)$ we have

\begin{itemize}
\item If $n\ge n_0,$ $\zeta\ge n\vre_n(\IE)$, and $e^{t/2}=n\sqrt{2c_2/\zeta}$, then
\eq{eq: first one}{
\phi(g_tq) \le \kappa \sqrt\zeta. 
}
\item If $n = \lfloor \kappa e^{t/2}\rfloor$, then
\eq{eq: second one}{
n\vre_n(\IE) \leq \kappa \phi(g_{t}q).
}
\end{itemize}
Moreover, 
 $\kappa, c_1, c_2, n_0$ may be taken to be
uniform for $q$ ranging over a compact subset of $\HH$ and $\gamma$
ranging over smooth curves of uniformly bounded length, with return
times to the curve bounded above and below. 
\end{prop}

\begin{proof}
We first claim that 
\eq{eq: defn epsn alt}{
\vre_n(\A) = \min \left\{\left|x_i - \IE^r
x_j\right|: 1 \leq i,j \leq d-1, |r| \leq n, (j,r) \neq (i,0)\right\}.
}
Indeed, if the minimum
in \equ{eq: defn epsn} is equal to $|\IE^kx_i - \IE^\ell x_j|$ with
$\ell \geq k \geq 1$, then the interval between
$\IE^kx_i$ and $\IE^\ell x_j$ does not contain discontinuities for
$\IE^{-k}$ (if it did the
minimum could be made smaller). This implies that $\IE^{-k}$ acts as
an isometry on this interval so that $|x_i - \IE^{\ell - k}x_j| = \vre_n(\IE)$.
Similarly, if the minimum in \equ{eq: defn epsn alt} is obtained for
$i,j$ and $r=-k\in[-n,0]$ then the interval between $x_i$ and
$\IE^{-k} x_j$ has no discontinuities of $\IE^k$, so that the same
value is also obtained for $|\IE^k x_i - x_j|$. Hence the minimum in
\equ{eq: defn epsn alt} equals the minimum in \equ{eq: defn epsn}.

Let $\bq \in \pi^{-1}(q)$ and let $n_0\ge 1$. 
Suppose the return times to
$\gamma$ along the vertical foliation
are bounded below and above by $c_1$ and $c_2$,
respectively. 
Making $c_2$ larger we can also assume that the total
variation in the vertical direction along $\gamma$ is no more than
$c_2$.  
Write $n \vre_n(\IE) = n|x_i - \IE^rx_j| \le\zeta$ and let $t$ be as in
\equ{eq: first one}.
Let $\sigma_i$ and $\sigma_j$ be the singularities of $\bq$ lying
vertically above $x_i$ and $x_j$. 
Let $\alpha$ be the path moving vertically from $\sigma_j$ along the
forward trajectory of $x_j$
until $\IE^rx_j$, then along $\gamma$ to $x_i$, and vertically up to $\sigma_i$.
Then $|x(\bq,\alpha)| = \vre_n(\IE) \le
\zeta/n$ and $|y( \bq,\alpha)| \leq c_2r + c_2  \leq 2nc_2$ for $n
\geq n_0$. Therefore, since $e^{t/2} = n\sqrt{2c_2/\zeta}$, we have
\[
\begin{split}
|x(g_t\bq,\alpha)| & = e^{t/2} |x(\bq,\alpha)| \leq \sqrt{2c_2\zeta}, \\
|y(g_t\bq,\alpha)| & = e^{-t/2} |y( \bq,\alpha)| \leq \sqrt{2c_2\zeta},
\end{split}
\]
so $\ell(\alpha, g_{t}\bq) \leq \kappa \sqrt\zeta$, where $\kappa = \sqrt{2c_2}.$
A shortest representative for $\alpha$ with respect to $g_{t}\bq$
is a concatenation $\bar\alpha$ of saddle connections. Since $\alpha$ travels
monotonically along both horizontal and vertical foliations of $\bq$,
a Gauss-Bonnet argument tells us that $\bar\alpha$ does the same, so that
the coordinates of its saddle connections have consistent signs. 
Hence the same bound holds for each of those saddle
connections, giving \equ{eq: first one}.

\medskip

Now we establish (\ref{eq: second one}). 
Let $\alpha$ be a saddle connection minimizing
$\ell(\cdot,g_t\bq)$, and write
$x_t = x(g_t\bq,\alpha)$ and 
$y_t = y(g_t\bq,\alpha)$.  Without loss of generality (reversing the
orientation of $\alpha$ if necessary)
we may assume that $x_t \ge 0$.
Minimality means
$$\phi = \phi(g_{t}q) = \max (x_t,|y_t|).$$
In $\bq$, the coordinates of $\alpha$ satisfy
$$
x_0 = e^{-t/2}x_t \le  e^{-t/2}\phi
$$
and 
$$
|y_0| = e^{t/2}|y_t| \le  e^{t/2}\phi.
$$
\begin{figure}[htp] 
\includegraphics{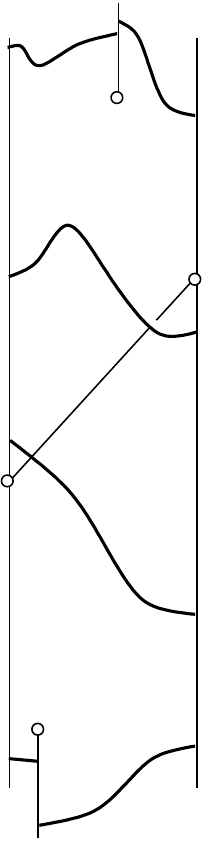}
\caption{The vertical strip $\UU$ minus rays $R_\sigma$ immerses
  isometrically in $S$.}
\name{figure: vstrip}
\end{figure}
Let $\UU$ be the strip $[0,x_0]\times\R$ in $\R^2$, and let $v\subset
\UU$ be the line segment connecting $v_-=(0,0)$ to
$v_+=(x_0,y_0)$. A neighborhood of $v$ in $\UU$ embeds in $S$ by a local
isometry that preserves horizontal and vertical directions. We can
extend this to an isometric immersion $\psi:\UU' \to
S$, where $\UU'$ has the following form: 
There is a discrete set $\hhat\Sigma \subset \UU \setminus int(v)$, and for
each $\sigma=(x,y)\in\hhat \Sigma$ a vertical ray $R_\sigma$ of the
form $\{x\}\times(y,\infty)$ (``upward pointing'') or 
$\{x\}\times(-\infty,y)$ (``downward pointing''), so that the rays are
pairwise disjoint,  disjoint from $v$, and $\UU' = \UU \setminus
\bigcup_{\sigma\in\hhat\Sigma} R_\sigma$ (see Figure \ref{figure: vstrip}).
The map $\psi$ takes
$\hhat\Sigma$ to $\Sigma$, and it is defined by extending the
embedding at each $p\in v$ maximally along the vertical line through
$p$ (in both directions) until the image encounters a singularity in $\Sigma$. 
(We include $v_-$ and $v_+$ in $\hat\Sigma$, and for these two points
delete both an upward and a downward pointing ray.) 


Let $\hhat\gamma$ be the preimage $\psi^{-1}(\gamma)$. This is a union
of arcs properly embedded in $\UU'$, and transverse to the vertical
foliation in $\R^2$. By definition of $c_1$ and $c_2$, each vertical
line in $\UU'$ is cut by $\hhat\gamma$ into segments of length at
least $c_1$ and at most $c_2$. Moreover the total vertical extent of
each component of $\hhat\gamma$ is at most $c_2$. 

Consider $\gamma_1$ the component of $\hhat\gamma$ that meets the
downward ray based at $v_+$ at the highest point $\hat r$. The other
endpoint $\hat p$ of $\gamma_1$ lies on some other ray $R_\sigma$.

The width of $\gamma_1$ is at most $x_0$, so the image points $r =
\psi(\hat r) $ and $p = \psi(\hat p)$ satisfy $|p-r|\le x_0$, with
respect to the induced transverse measure on $\gamma$. We now check
that $p$ and $r$ are images of discontinuity points of $\IE$, by
controlled powers of $\IE$.

By choice of $\gamma_1$, the upward leaf emanating from $r$ encounters
the singularity $\psi(v_+)$ before it returns to $\gamma$, and hence
$r$ itself is a discontinuity point $x_i$. 

For $p$, Let us write $\sigma = (x,y)$ and $p=(x,y')$. Suppose first
that $y_0\ge 0$. There are now two cases. If $R_\sigma$ lies above $v$ (and hence
is upward pointing): we have $y'\ge y \ge 0$, and moreover (since the
vertical variation of $\gamma_1$ is bounded) $y' \le y_0+c_2$. 
The segment of
$R_\sigma$ between $\sigma $ and $p$ is cut by $\hhat \gamma$
(incident from the right) into at
most  $(y'-y)/c_1$ pieces, and this implies that there is some $k\ge 0$
bounded by
$$
k \le \frac{y_0+c_2}{c_1} \le \frac{e^{t/2}\phi + c_2}{c_1}
$$
such that $p= T^k x_j$ for some discontinuity $x_j$.

If $R_\sigma$ lies below $v$ and is downward pointing: we have $y' \le
y \le y_0$ and $y' \ge y_0-2c_2$, so that by the same logic there is $k\ge
0$ with 
$$
k \le \frac{2c_2}{c_1}
$$
such that $T^k p = x_j$ for some discontinuity $x_j$. 

Hence in either case we have 
\eq{eq: x0 bound}{
|T^m x_j - x_i| \le x_0 \le e^{-t/2}\phi
}
where $-2c_2/c_1 \le m \le (y_0+c_2)/c_1$. 

If $y_0 <0$ there is a similar analysis, yielding the bound 
\equ{eq: x0 bound} where now $(y_0-2c_2)/c_1 \le m \le c_2/c_1$.

Noting that $\phi < 1$ by
area considerations, 
if we take $n = \lfloor
\kappa e^{t/2}\rfloor$, where $\kappa  = 1+(1+2c_2)/c_1$, then we
guarantee  $|m| \le n$, and hence  get
$$
n\vre_n \le \kappa e^{t/2} e^{-t/2}\phi \le \kappa\phi.
$$

\medskip

\end{proof}

\section{Mahler's question for lines}\name{section:
inheritance lines}
In this section we will derive Theorem \ref{cor: inheritance, lines}
from Theorem \ref{thm: sullivan2} and earlier results
of \cite{KMS, Masur Duke, with Dima, bad}. 
We will need the following:
\begin{prop}
\name{lem: horocycles and circles}
For any $|\theta| < \pi/2$, there is a bounded subset $\Omega \subset G$
such that for any $t \geq 0$, and any $q \in \HH$, there is $w \in
\Omega$ such that $g_t h_{-\tan \theta} q = w g_t r_{\theta} q.$ 
\end{prop}
\begin{proof}
Let
$$x = \left(\begin{matrix}1/\cos \theta & 0 \\ -\sin \theta & \cos
\theta \end{matrix} \right) \in G. 
$$
Then $g_t x g_{-t}$ converges in
$G$ as $t \to \infty$, and we set $\Omega = \{g_txg_{-t} : t \geq
0\}$. 
Since $xr_\theta = h_{-\tan \theta}$ and
$g_t h_{-\tan \theta} q = g_t x
r_\theta q = g_t x
g_{-t} g_t r_\theta q$, the claim follows.  
\end{proof}

\begin{proof}[Proof of Theorem \ref{cor: inheritance, lines}]
Let $(\A, \B)$ be positive, let $\UU$ be a neighborhood
of $(\A, \B)$ in $\R^d_+ \times \R^d$, let $\bq: \UU \to \HH$ as in
Theorem \ref{thm: sullivan2}, let $q = \pi \circ \bq$ where $\pi: \til
\HH \to \HH$ is the natural projection, and let $\vre_0>0$ so that
$\A(s)=\A+s\B \in \UU$ for all $s \in (-\vre_0, \vre_0)$. Making $\vre_0$
smaller if necessary, let $\gamma$
be a judicious curve for $q$ such that $\IE_\sigma(\A(s)) = \IE(h_sq,
\gamma)$ for all $s \in (-\vre_0, \vre_0)$. By Theorem \ref{thm:
sullivan2} and Proposition \ref{lem: relating 
lengths and disconts}, $\A(s)$ is of recurrence (resp. bounded) type
if and only if there is a compact subset $K \subset \HH$ such that
$\{t >0 : g_th_sq \in K\}$ is unbounded (resp., is equal to $(0,
\infty)$). The main result of \cite{KMS} is that for any $q$, for
Lebesgue-a.e. $\theta \in (-\pi, \pi)$ there is a compact $K
\subset 
\HH$ such that $\{t>0: g_tr_{\theta}q \in K\}$ is unbounded. Thus (a)
(with $\mu$ equal to Lebesgue measure)
follows via Proposition
\ref{lem: horocycles and circles}. For a general measure $\mu$, the
statement will follow from
Corollary \ref{cor: strengthening} below. 
Similarly (b) follows from \cite{with Dima}
for $\mu$ equal to 
Lebesgue measure, and from \cite{bad} for a general decaying Federer
measure, and (c) follows from \cite{Masur Duke}.
\end{proof}

\section{Quantitative nondivergence for horocycles}
\name{section: nondivergence}
In this section we will recall a quantitative nondivergence result for
the horocycle flow, which 
is a variant of results in \cite{with Yair}, and will be crucial for
us. The theorem was stated
without proof in \cite[Prop. 8.3]{bad}. At the end of the section we
will use it to obtain a strengthening 
of Theorem \ref{cor: inheritance, lines}(a).

Given positive constants $C, \alpha, D$, we say that a regular finite Borel measure $\mu$ on $\R$ is {\em
$(C, \alpha)$-decaying and $D$-Federer} if \equ{eq: decaying and
federer} holds for all $x \in \supp \, \mu$ and all $0 < \vre, r <1.$ 
For an interval
$J=B(x,r)$ and $c>0$ we write $cJ = B(x, cr)$. Let $\HH_1$ and $K_\vre$
be as in \S\ref{section: saddles}, and let $\til \HH_1 = \pi^{-1}(\HH_1)$.

\begin{thm}
\name{thm: nondivergence, general}
Given a stratum $\HH$ of translation surfaces\footnote{The 
result is also valid (with identical proof) in the more general setup
of quadratic differentials.}, there are positive
constants $C_1, \lambda, \rho_0$, such that for any 
$(C,\alpha)$-decaying and $D$-Federer measure $\mu$ on
an interval $B\subset \R$, the following holds. Suppose $J \subset \R$
is an interval with $3J
\subset B$, 
$0<\rho
\leq \rho_0, \, \bq
\in \til{\HH}_1$, and suppose
\eq{eq: condn on J}{\forall \delta \in \LL_{\bq}, \ \sup_{s \in J}\,
\ell(\delta, h_s\bq) \geq \rho.}
Then for any $0<\vre<\rho$:
\begin{equation}
\label{eq - precise yet long}
\mu\left( \{s \in J: h_s \pi(\bq) \notin K_{\vre} \right \}) \leq
C' \left(\frac{\vre}{\rho}\right)^{\lambda \alpha}\mu(J),
\end{equation}
where $C' = C_1 2^\alpha C D$. 
\end{thm}

\begin{proof}
The proof is similar
to that of \cite[Thm.\ 6.10]{with Yair}, but with the assumption that
$\mu$ is Federer
substituting for condition (36) of that paper.
To avoid repetition we give the proof 
making reference to \cite{with Yair} when necessary.

Let $\lambda, \rho_0, C_1$ substitute for $\gamma, \rho_0, C$ as in
\cite[Proof of Thm. 6.3]{with Yair}. 
For an interval $J \subset  \R$, let $|J|$ denote its length.
For a function $f : \R \to \R_+$ and $\vre>0$, let
$$J_{f, \vre} = \{x \in J: f(x)<\vre\} \ \ \mathrm{and } \ \|f\|_J =
\sup_{x \in J} f(x).$$ 
For $\delta \in \LL_\bq$ let $\ell_{\delta}$ be the function
$\ell_{\delta}(s) = \ell(\delta, h_s\bq)$.  
Suppose, for $\bq \in \til \HH, \, \delta \in \LL_{\bq}$ and an interval
$J$, that $\|\ell_{\delta} \|_{J} \geq \rho.$ An elementary computation
(see \cite[Lemma
4.4]{with Yair}) shows that  $J_{\vre} = J_{\ell_{\delta}, \vre}$ is a
subinterval of $J$ and
\begin{equation}
\label{eq: ratio of lengths}
|J_{\vre}| \leq \frac{2 \vre}{\rho} |J|.
\end{equation}

Suppose that $\mu$ is
$(C,\alpha)$-decaying and $D$-Federer, and $\supp \, \mu \cap
J_{\vre} \neq \varnothing$.
Let $x \in \supp \, \mu \cap J_{\vre}$.
Note that $J_{\vre} \subset B(x, |J_{\vre}|)$ and $B(x, |J|) \subset
3J.$
One has
\[
\begin{split}
\mu(J_{\vre}) & \leq \mu(B(x, |J_{\vre}|)) \\
& \stackrel{\mathrm{decay, \ } \equ{eq: ratio of
lengths}}{\leq} C \left(\frac{2
\vre}{\rho} \right)^{\alpha} \mu(B(x, |J|)) \\
& \leq 2^{\alpha} C \left(\frac{\vre}{\rho} \right)^{\alpha} \mu(3J)
\leq C'' \left(\frac{\vre}{\rho} \right)^{\alpha} \mu(J),
\end{split}
\]
where $C'' = 2^\alpha CD$. 
This shows that if $J$ is an interval, $\bq \in \til \HH$, and $\delta \in
\mathcal{L}_{\bq}$ is such that $\|\ell_{\delta}\|_J \geq \rho$, then for
any $0<\vre <\rho$,
\begin{equation*}
\frac{\mu(J_{\ell_{\delta}, \vre})}{\mu(J)} \leq C''
\left( \frac{\vre}{\rho} \right)^{\alpha}.
\end{equation*}

Now to obtain \equ{eq - precise yet long}, define $F(x)=C''x^{\alpha}$
and repeat the proof of \cite[Theorem
6.3]{with Yair}, but using $\mu$ instead of Lebesgue measure on $\R$
and using
\cite[Prop. 3.4]{with Yair} in place of \cite[Prop. 3.2]{with
Yair}. 
\end{proof}

\begin{cor}
\name{cor: return to compacts}
For any stratum $\HH$ of translation surfaces and any $C, \alpha, D$
there is a compact $K \subset \HH_1$ such that for any $q \in \HH_1$, any
unbounded $\mathcal{T} \subset \R_+$ and
any $(C, \alpha)$-decaying and $D$-Federer measure $\mu$ on an interval
$J \subset \R$, for $\mu$-a.e. $s \in J$ there is a sequence $t_n \to
\infty, \, t_n \in \mathcal{T}$ such that $g_{t_n}h_sq \in K$. 
\end{cor}

\begin{proof}
Given $C, \alpha, D$, let $\lambda, \rho_0, C'$ be as in Theorem
\ref{thm: nondivergence, 
general}. Let $\vre$ be small enough so that 
$$C' \left (\frac{\vre}{\rho_0} \right)^{\lambda \alpha} <1,
$$
and let $K=K_{\vre}$. Suppose to the contrary that for some 
$(C,\alpha)$-decaying and $D$-Federer measure $\mu$ on some interval
$J_0$ we have  
$$\mu(A) > 0, \ \ \mathrm{where} \ A=\{s \in J_0 :
\exists t_0 \, \forall t \in \mathcal{T} \cap (t_0, \infty), \, g_th_sq \notin K\}.$$
Then there is $A_0 \subset A$ and $t_0>0$ such that $\mu(A_0)>0$ and
\eq{eq: property A0}{
s \in A_0, \ t \in \mathcal{T} \cap (t_0, \infty) \ \implies \ g_th_sq \notin K.}
By a general density theorem, see e.g. \cite[Cor. 2.14]{Mattila},
there is an interval $J$ with $3J \subset J_0$ such that 
\eq{eq: from density}{
\frac{\mu(A_0 \cap J)}{\mu(J)} > C' \left (\frac{\vre}{\rho_0} \right)^{\lambda 
\alpha}.}
We claim that by taking $t > t_0$ sufficiently large 
we can assume that for all $\delta \in \LL_{\bq}$ there is $s \in J$
such that 
$\ell(\delta, g_th_s\bq) 
\geq \rho_0.$ 
This will guarantee that \equ{eq: condn on J} holds for the horocycle
$s \mapsto g_th_sq = h_{e^ts}g_tq$, and conclude the proof since \equ{eq:
property A0} and \equ{eq: from 
density} contradict \equ{eq - precise yet long} (with $g_{t}q$ in
place of $q$).

It remains to prove the claim. Let $\zeta = \phi(q)$
so that for any $\delta \in \LL_{\bq}$, 
$$\ell(\delta, \bq) = \max \left
\{|x(\delta, \bq)|, |y(\delta, \bq)| \right\}
\geq \zeta.$$
If $|x(\delta, \bq )| \geq \zeta$, then 
$ |x(\delta, g_t\bq)| = e^{t/2}
|x(\delta, \bq)| \geq \zeta e^{t/2},$ and if 
$|y(\delta, \bq)| \geq \zeta$ then the function 
$$s \mapsto x(\delta, g_t h_s\bq) = e^{t/2}\left(x(\delta,
\bq)+sy(\delta, \bq)  \right)$$
has slope $|e^{t/2}y(\delta, \bq)| \geq e^{t/2} \zeta$, hence 
$\sup_{s \in J} 
|x(\delta, g_t h_s \bq)| \geq \zeta e^{t/2} 
\frac{|J|}{2}.$
Thus the claim holds when $\zeta e^{t/2} \geq  \max \left\{\rho_0, \rho_0
\frac{|J|}{2} \right \}.$ 
\end{proof}

This yields a strengthening of Theorem \ref{cor:
inheritance, lines}(a). 
\begin{cor}
\name{cor: strengthening}
Suppose $(\A, \B)$ is positive, and write $\A(s) = \A+s\B$. There is
$\vre_0>0$ such that 
given $C, \alpha, D$ there is $\zeta>0$ such that if $\mu$ is $(C,
\alpha)$-decaying and $D$-Federer, and $\supp \, \mu \subset (-\vre_0,
\vre_0)$, then for $\mu$-almost every $s$, $\limsup n
\vre_n(\A(s)) \geq \zeta.$

\end{cor}

\begin{proof}
Repeat the proof of Theorem \ref{cor: inheritance, lines}, using
Corollary \ref{cor: return to
compacts} and Proposition 
\ref{lem: relating lengths and disconts}
instead of \cite{KMS}. 
\end{proof}

\section{Mahler's question for curves}\name{section: mahler curves}
In this section we prove Theorems \ref{thm:
mahler curve} and \ref{cor: ue on curves} by deriving
them from a 
stronger statement.

\begin{thm}
\name{thm: curves}
Let $J \subset \R$ be a compact interval, let $\beta: J \to \R^d_+$ be a
$C^2$ curve, let $\mu$ be a decaying Federer measure on
$J$, and suppose that for every $s_1, s_2 \in J$, $(\beta(s_1),
\beta'(s_2))$ is a positive pair. Then there
is $\zeta>0$ such that for $\mu$-a.e. $s \in J, $ $
\limsup_{n \to \infty} n\vre_n(\beta(s)) \geq
\zeta$. 
\end{thm}

\begin{proof}[Derivation of Theorem \ref{cor: ue on curves} from
Theorem \ref{thm: curves}]
If Theorem \ref{cor: ue on curves} is false then there is $A \subset I$
with $\mu(A)>0$ such that for all $s \in A$, $\beta(s)$ is not of
recurrence type but 
$(\beta(s), \beta'(s))$ is positive. Let $s_0 \in A \cap \supp \, \mu$
so that $\mu(A \cap J) > 0$ for any open interval $J$
containing $s_0$. Since the set
of positive pairs is open (Corollary \ref{cor: realizable open}), there is
an open $J$ containing $s_0$ such that $(\beta(s_1), \beta'(s_2))$ is
positive for every $s_1, s_2 \in J$, so Theorem \ref{thm: curves}
implies that $\beta(s)$ is of recurrence type for almost every $s \in
J$, a contradiction.
\end{proof}

\begin{proof}[Proof of Theorem \ref{thm: mahler curve}]
Let $\A(x)$ be as in \equ{eq: mahler curve} and let $\|\cdot \|_1$ be
the 1-norm on $\R^d$. Since unique ergodicity is unaffected by
dilations, it is 
enough to verify the conditions of Theorem \ref{cor: ue on curves} for
the permutation $\sigma(i)=d+1-i$, a
decaying Federer measure $\mu$, and for 
$$\beta(s) = \frac{\A(s)}{\|\A(s)\|_1} = \frac{1}{s+\cdots +s^d}
\left(s, s^2, \ldots, s^d\right).$$ 
 For any connection $(i,j,m)$ the
set  
$\left \{\A \in \Delta: (i,j,m) \in \LL_\A \right\}$
is a proper affine subspace of $\R^d_+$ transversal to $\{(x_1,
\ldots, x_d): \sum x_i=1\}$, and since $\beta(s)$ is
analytic and not contained in any such affine subspace, the set 
$\{s \in I : \beta(s) \mathrm{\ has \ connections}\}$
is countable, so $\beta(s)$ is without connections for
$\mu$-a.e. $s$. 

Letting $R=R(s) = s+\cdots+s^d$, we have 
$$\beta'(s) = \frac{1}{R^2} \left(\gamma_1(s), \ldots, \gamma_d(s) 
\right), \ \ \ \mathrm{where} \ \ \gamma_i(s) = \sum_{\ell = i}^{i+d-1}
\left(2i-\ell-1 \right) s^{\ell}.
$$
Then for
$j=1, \ldots, d-1$ and $\ell =1, \ldots, j+d-1$, setting $a= \max\{1,
\ell+1-d\},\, b=\min\{j, \ell\}$ we find:
\[
\begin{split}
R^2y_j & = \sum_{i=1}^j \gamma_i(s) =\sum_{\ell=1}^{j+d-1} \left(\sum_{i=a}^b (2i-\ell-1) \right)
s^\ell \\ 
& = \sum_{\ell =1}^{j+d-1}\left[ b(b+1)-a(a-1) -
(b-a+1)(\ell+1)\right] s^{\ell}.
\end{split} 
\]
Considering separately the 3 cases $1 \leq \ell \leq j,\ j < \ell \leq
d, \ d<\ell$ one sees that in every
case $y_j < 0$. Our choice of $\sigma$ insures that for $j = 1,
\ldots, d-1$,
$y'_j = -y_{d-j}>0 $. This implies via \equ{eq: defn L} that $L<0$ on
$I$, thus 
for all $s$ for which $\beta(s)$ is without
connections, $\left(\beta(s), -\beta'(s)\right)$ is positive. Define
$\hat{\beta}(s)=\beta(-s)$, so that $\left(\hat{\beta}(s), \hat{\beta}'(s)\right)
= \left(\beta(-s), -\beta'(-s)\right)$ is positive for a.e. $s<0$. Thus Theorem
\ref{cor: ue on curves} applies to $\hat{\beta}$, proving the claim. 
\end{proof}

\begin{proof}[Proof of Theorem \ref{thm: curves}]
Let $\HH$ be the stratum corresponding to $\sigma$. For a
$(C, \alpha)$-decaying and $D$-Federer measure $\mu$, let $C', \rho_0, \lambda$ be the
constants as in Theorem \ref{thm: nondivergence, general}. Choose  
$\vre>0$ small enough so that 
\eq{eq: choice of epsilon}{
B = C' \left(\frac{\vre}{\rho_0}
\right)^{\lambda \alpha} <1,}
and let $K=K_{\vre}.$

By making $J$ smaller if necessary, we can assume that for all $s_1,
s_2 \in J$, there is a translation surface $q(s_1,s_2) = \pi\circ \bq(\beta(s_1),
\beta'(s_2))$ corresponding to the positive pair $(\beta(s_1),  
\beta'(s_2))$ via Theorem \ref{thm: sullivan2}. That is
$Q=\{q(s_1, s_2): s_i \in J\}$ is a bounded subset of $\HH$ and
$q(s_1, s_2)$ depends continuously on $s_1, s_2$. By appealing to
Corollary \ref{cor: same gamma}, we can also assume that there is a
fixed curve $\gamma$ so that $\IE_\sigma(\beta(s_1)) = \IE(q(s_1,
s_2), \gamma),$ for $s_1, s_2 \in J$. Define $q(s) = q(s,s)$. 
By rescaling we may assume with no loss of
generality that $q(s) \in \HH_1$ for all $s$, and by
making $K$ larger let us assume that $Q \subset K.$ 
By continuity, 
the return times to $\gamma$ along vertical leaves for
$q(s_1, s_2)$
are uniformly bounded from above and below, and the length of $\gamma$
with respect to the flat structure given by $q(s_1, s_2)$ is uniformly
bounded. 

We claim that there is $C_1$,
depending only on $Q$, such that for any interval $J_0 \subset \R$
with $0 \in J_0$,
any $t>0$, any $q \in Q$, any $\bq \in
\pi^{-1}(q)$ and any $\delta \in
\LL_{\bq}$, we have
$$\sup_{s \in J_0} \ell(\delta, g_t h_{s} \bq) \geq C_1 |J_0| e^{t/2}.$$
Here $|J_0|$ is the length of $J_0$. 

Indeed, let 
$\theta = \inf\{\phi(q): q  \in Q\},
$
which is a positive number since $Q$ is bounded. 
Let $C_1 =  \min\left\{\frac{\theta}{|J_0|},
\frac{\theta}{2} \right\}$, and let $q \in Q, \, \bq \in \pi^{-1}(q),
\, \delta \in \LL_{\bq}$. Then 
$\max \left
\{|x(\delta, \bq)|, |y(\delta, \bq)| \right\} \geq \theta.$
Suppose
first that $|x(\delta, \bq )| \geq \theta$, then 
$$\sup_{s \in J_0}\ell(\delta, g_th_s\bq) \geq |x(\delta, g_t\bq)| = e^{t/2}
|x(\delta, \bq)| \geq \theta e^{t/2} \geq C_1 |J_0| e^{t/2}.$$
Now if $|y(\delta, \bq)| \geq \theta$ then the function 
$$s \mapsto x(\delta, g_t h_s\bq) = e^{t/2}\left(x(\delta,
\bq)+sy(\delta, \bq)  \right)$$
has slope $|e^{t/2}y(\delta, \bq)| \geq e^{t/2} \theta$, hence 
$$\sup_{s \in J_0} \ell(\delta, g_th_s\bq) \geq \sup_{s \in J_0}
|x(\delta, g_t h_s \bq)| \geq e^{t/2} 
\theta |J_0|/2 \geq C_1 |J_0| e^{t/2}.$$
This proves the claim. 

\medskip 


For each $s_0, s \in J$ let $\A^{(s_0)}(s) = 
\beta(s_0)+\beta'(s_0)(s-s_0)$ be the linear approximation to $\beta$
at $s_0$.
Using the fact that $\beta$ is a $C^2$-map, there is $\til C$ such
that 
\eq{eq: small distance}{\max_{s \in J_0} \| \beta(s) - \A^{(s_0)}(s)
\| 
< \til C | J_0|^2
}
whenever $J_0 \subset
J$ is a subinterval 
centered at $s_0$. 

Let $\kappa$ and $c_2$ be as in Proposition \ref{lem: relating
lengths and disconts}, let $\vre$ be as chosen in \equ{eq: choice of
epsilon}, let
$$
C_2 = \frac{\rho_0}{2C_1}, \ \ \ 
\zeta_1 <\left(\frac{\vre}{\kappa} \right)^2 \ \ \ \mathrm{and} \ \ 
\zeta = \zeta_1 \cdot \frac{c_2}{d^2\til C}.
$$

If the theorem is false then $\mu(A)>0$, where 
$$A = \{s \in J: \limsup n\vre_n(\beta(s)) < \zeta \}.$$
Moreover there is $N$ and
$A_0 \subset A$ such that $\mu(A_0)>0$ and 
\eq{eq: uniformly small}{
n \geq N, \, s \in A_0 \ \implies \ n\vre_n(\beta(s)) < \zeta.
}
Using
\cite[Cor. 2.14]{Mattila} let $s_0 \in A_0$ be a density point, so that
for any sufficiently small interval $J_0$ centered at $s_0$ we have 
\eq{eq: density argument}{
\mu(A_0 \cap J_0) >B \mu(J_0),
}
where $B$ is as in \equ{eq: choice of epsilon}. 

For $t>0$ we will write 
$$c(t) = C_2 e^{-t/2} \ \  \ \mathrm{and}\ \ \ J_{t} =
B(s_0,c(t)).$$
Let $\A(s) = \A^{(s_0)}(s)$ and let $\til q(s)=h_{s-s_0}q(s_0)$, which is the surface
$\pi \circ \bq (\A(s), \beta'(s_0))$. 
The trajectory $s
\mapsto g_t \til q(s) = g_t h_{s-s_0}q(s_0) = h_{e^ts} g_th_{-s_0}
q(s_0)$ is a horocycle path, which, by the claim and the choice of
$C_2$, satisfies \equ{eq: condn
on J} with $\rho = \rho_0$ and $J=J_t$ for all $t>0$. Therefore 
\eq{eq: measure estimate 2}{
\mu \left\{s \in J_t : g_t \til q(s) \notin K \right \} \leq B \mu (J_t).
}

Now for a large $t>0$ to be specified below, let 
\eq{eq: defn n2}{
 n_1 =
\sqrt{\frac{\zeta_1}{2c_2}}e^{t/2} \ \ \ \mathrm{and} \ \ n_2 = \frac{2c_2}{d^2
\til C} n_1.
}
By making $\til C$ larger we can assume that $n_2 < n_1$. 
For large enough $t$ we will have $n_2 > n_0$ (as in Proposition
\ref{lem: relating lengths and disconts}) and $n_2 > N$ (as in
\equ{eq: uniformly small}). 


We now claim 
\eq{eq: remains to prove}{
s \in J_t, \ n_2 \vre_{n_2} (\beta(s))
< \zeta \ \implies \ n_1\vre_{n_1}(\A(s)) < \zeta_1.
}
Assuming this, note that by \equ{eq:
first one}, \equ{eq: defn n2} and the choice of $\zeta_1$, if $n_1
\vre_{n_1}(\A(s)) < \zeta_1$ then $g_t \til q(s) 
\notin K$. Combining \equ{eq: uniformly small} and \equ{eq: remains to
prove} we see that $A_0 \cap J_t \subset \{s \in J_t: g_t \til q(s)
\notin K\}$ for all large enough $t$. Combining this with \equ{eq:
density argument} we find a contradiction to \equ{eq: measure estimate
2}. 

\medskip

It remains to prove \equ{eq: remains to prove}. Let $r \leq n_2$, 
let $\IE= \IE_\sigma(\A(s)), \ \SSS= \IE_\sigma(\beta(s))$, let $x_i,
x_j$ be discontinuities of $\IE$ and let $x'_i, x'_j$ be the
corresponding discontinuities of $\SSS$. By choice of $\til C$ we have   
$$\| \beta(s) - \A(s)\| < \til C e^{-t},
$$
where $\| \cdot \|$ is the max norm on $\R^d$. 
Suppose first 
that $x_i$ and $x'_i$ have the same itinerary
under $\IE$ and $\SSS$ until the $r$th iteration; i.e. $\IE^kx_i$ is
in the $\ell$th interval of continuity of $\IE$ if and only if $\SSS^k
x'_i$ is in the $\ell$th interval of continuity of $\SSS$ for $k \leq
r$. Then one sees from \equ{eq: disc points2} and \equ{eq: defn iet}  that 
$$|\IE^r x_i - \SSS^r x'_i| \leq \sum_0^r d^2 \| \beta(s) - \A(s)\|.$$
Therefore, using \equ{eq: defn n2}, 
\[
\begin{split}
\left| x_i - \IE^rx_j - (x'_i - \SSS^r x'_j) \right | & \leq
\left| x_i - x'_i \right | + \left| \IE^rx_j - \SSS^r x'_j \right | \\
& \leq 2 \sum_1^r d^2 \| \beta(s) - \A(s) \| \\
& \leq 2 d^2 n_2 \til C e^{-t} = \frac{\zeta_1}{2n_1}.
\end{split}
\]
If $n_1\vre_{n_1}(\A(s)) \geq \zeta_1$ then by the triangle inequality
we find 
that $|x'_i - \SSS^rx'_j| < \zeta_1/2n_1>0$. In particular this shows
that for all $i$, $x_i$ and $x'_i$ do have the same itinerary under
$\IE$ and $\SSS$, and moreover
\[
\begin{split}
\vre_{n_2}( \beta(s)) & \geq \vre_{n_2}(\A(s)) - \frac{\zeta_1}{2n_1} \\
& \geq \vre_{n_1}(\A(s)) - \frac{\zeta_1}{2n_1} \\
&  > \frac{\zeta_1}{2n_1} = \zeta_1 \frac{n_2}{2n_1} \cdot \frac{1}{n_2} = \frac{\zeta}{n_2},
\end{split}
\]
as required. 
\end{proof}
\section{Real REL}
This section contains our results concerning the real REL
foliation, whose definition we now recall. Let $\bq$ be a 
marked translation surface of type $\br$, where $k$ (the number of
singularities) is at least 2. Recall that $\bq$ determines a 
cohomology class $\hol(\bq)$ in $H^1(S, \Sigma ; \R^2)$, where for a relative
cycle $\gamma \in H_1(S, \Sigma),$ the value $\bq$ takes on
$\gamma$ is $\hol(\gamma, \bq)$. Also recall that there are open sets
$\til \HH_\tau$ in $\til \HH(\br)$, corresponding to a given triangulation of
$(S, \Sigma)$, such that the map hol restricted to $\til \HH_\tau$ endows
$\til\HH$ with a linear manifold structure. Now recall the map Res as in 
\equ{eq: defn Res}, let $V_1$ be the first summand in the splitting \equ{eq:
splitting}, and let $W = V_1 \cap \ker \,\Res, $ so that $\dim W =
k-1,$ where $k = |\Sigma|$. The REL foliation is modeled on $\ker \,
\Res$, the real foliation is modeled on $V_1$, and the real REL
foliation is modeled on $W$. 
That is, a ball $\mathcal{U} \subset \til \HH_\tau$ provides a product
neighborhood for these foliations, where $\bq, \bq_1 \in \mathcal{U}$
belong to the same plaque for REL, real, or real REL, if
$\hol(\bq)-\hol(\bq_1)$ belongs respectively to $\ker \Res, V_1,$ or
$W$. 

Recall that $\HH$ is an orbifold and the orbifold cover $\pi: \til \HH
\to \HH$ is defined by taking a quotient by the $\Mod(S,
\Sigma)$-action. 
Since hol is equivariant with respect to the action of the group $\Mod(S,
\Sigma)$ on $\til \HH$ and $H^1(S, \Sigma; \R^2)$, and (by naturality
of the splitting \equ{eq:
  splitting} and the sequence \equ{eq: defn Res}) the subspaces $V_1, W,$ 
and $\ker \, \Res$ are $\Mod(S, \Sigma)$-invariant, the foliations
defined by these subspaces on $\til \HH$ descend naturally to $\HH$. 
More precisely leaves in $\til \HH$ descend to `orbifold leaves' on
$\HH$, i.e. leaves in $\til \HH$ map to immersed sub-orbifolds in
$\HH$. In order to avoid dealing with orbifold foliations, we 
pass to a finite cover $\hat{\HH}  
\to \HH$ which is a manifold, as explained in \S \ref{subsection: strata}. 

\begin{prop}
\name{prop: independent of marking}
The REL and real REL
leaves on both $\til \HH$ and $\hat \HH$ have a well-defined translation structure. 

\end{prop}  

\begin{proof}
 A translation structure on a leaf $\mathcal{L}$
amounts to saying that there is a fixed vector space $V$ and an 
atlas of charts on $\mathcal{L}$ taking values in $V$, so that transition
maps are translations. We take $V = \ker \, \Res$ for the REL leaves,
and $V=W$ for the real REL leaves. For each 
$\bq_0 \in \til \HH$, the atlas is obtained by taking 
the chart $\bq \mapsto \hol(\bq)-\hol(\bq_0)$. By the definition of
the corresponding foliations, these are homeomorphisms in a
sufficiently small neighborhood of $\bq_0$, and the fact that
transition maps are translations is immediate. In order to check that
this descends to $\hat \HH$, let $\Lambda \subset \Mod (S, \Sigma)$ be
the finite-index torsion free subgroup so that $\hat \HH = \til \HH
/\Lambda$, and let $\varphi \in \Lambda$. 
We need to show that if  $\hat{\bq} = \bq \circ \varphi, \ \hat{\bq}_0 =
\bq_0 \circ \varphi,$  then $\hol(\bq)- \hol(\bq_0) = \hol(\hat{\bq}) -
\hol(\hat{\bq}_0)$. Since $\hol$ is $\Mod(S, \Sigma)$-equivariant,
this amounts to checking that $\varphi$ acts trivially on $\ker \,
\Res$. 
Invoking \equ{eq: defn Res}, this follows from our convention that any $\varphi \in \Mod(S,
\Sigma)$ fixes each point of $\Sigma$, so acts trivially on
$H^0(\Sigma)$.  
\end{proof}

It is an interesting question to understand the geometry of individual
leaves. For the REL foliation this is a challenging problem, but for
the real REL foliation, our main theorems give a complete answer. 
Given a 
marked translation surface $\bq$ we say that a saddle connection $\delta$ for
$\bq$ is {\em horizontal} if $y(\bq, \delta)=0$. Note that for a generic
$\bq$ there are no horizontal saddle connections, and in any stratum
there is a uniform upper bound on their number. 

\begin{thm}\name{thm: real rel main}
Let $\bq \in \til \HH$ and let $\mathcal{V} \subset W$ such that for
any $\cc \in \mathcal{V}$ there is a 
path $\left\{\cc_t\right\}_{t \in [0,1]}$ in $\mathcal{V}$ from 0 to $\cc$
such that for any horizontal saddle connection
for $\bq$ and any $t \in [0,1]$, $\hol(\bq, \delta) + \cc_t (\delta)
\neq 0$. 
Then there is a continuous map $\psi: \mathcal{V} \to \til \HH$ such
that for any $\cc \in \mathcal{V}$, 
\eq{eq: desired}{\hol(\psi(\cc)) = \hol(\bq) + (\cc,0),}
and the horizontal foliations of $\psi(\cc)$ is the same as that of
$\bq$. Moreover the image of $\psi$ is contained in the REL leaf of
$\bq$. 


\end{thm}

\begin{proof}
Let $\FF$ and $\GG$ be the vertical and horizontal foliations of
$\bq$ respectively. We will apply Theorem \ref{thm: hol homeo}, reversing the roles
of the horizontal and vertical foliations of $\bq$; that is we use
$\GG$ in place of $\FF$. To do this, we will check that the map which
sends $\cc \in \mathcal{V}$ to  
$(x(\bq)+\cc, y(\bq)) \in H^1(S, \Sigma)^2
$
has its image in 
$\BBB(\GG) \times \AAA(\GG)$ (notation as in Theorem \ref{thm: hol
  homeo}), and thus 
$$\psi(\cc) = \hol^{-1}(x(\bq)+\cc, y(\bq))$$
is continuous and satisfies \equ{eq: desired}. 

Clearly $y(\bq)$, the cohomology class represented by $\GG$, is in
$\AAA(\GG)$. To check that $x(\bq)+\cc \in \BBB(\GG)$, we need to show
that for any element $\delta \in H^\GG_+$, $x(\bq, \delta) +
\cc(\delta)>0$. To see this, we treat separately the cases when
$\delta$ is a horizontal saddle connection, and when 
$\delta$ is represented by a foliation cycle, i.e. an
element of $H_1(S, \Sigma)$ which is in the image of $H_1(S)$ (and
belongs to the convex cone spanned by the asymptotic cycles). 
If $\delta$ is a foliation cycle, since
$\cc \in W \subset \ker \, \mathrm{Res}$, the easy direction in
Theorem \ref{thm: sullivan1} implies
$$
x(\bq, \delta) + \cc(\delta) = x(\bq, \delta)>0.
$$
If $\delta$ is represented by a saddle connection, let $\{\cc_t\}$ be
a path from 0 to $\cc$ in $\mathcal{V}$. Again, by the easy direction
in Theorem \ref{thm: sullivan1}, $x(\bq, \delta)>0$, and the function 
$x(\bq, \delta)+\cc_t(\delta)$ is a continuous function of $t$, which
does not vanish by hypothesis. This implies again that $x(\bq, \delta)
+\cc(\delta)>0$.

To check $\psi(\cc)$ is contained in the real
REL leaf of $\bq$, it suffices to note that according to \equ{eq:
  desired}, in any local chart provided
by hol, $\psi(\cc_t)$ moves along plaques of the foliation. 
\end{proof}

Since the leaves are modelled on $W$, taking $\mathcal{V}=W$ we obtain:
\begin{cor}
If $\bq$ has no horizontal saddle connections, then there is a
homeomorphism $\psi: W \to \til \HH$ 
onto the leaf of $\bq$. 
\end{cor}

Moreover the above maps are compatible with the transverse 
structure for the real REL foliation. Namely, let $\psi_\bq$ be the
map in Theorem \ref{thm: real rel main}. We have:

\begin{cor}\name{cor: foliation compatible}
For any $\bq_0 \in \til \HH$ there is a neighborhood
$\mathcal{V}$ of 0 in $W$, and a submanifold $\mathcal{U}' \subset
\til \HH$ everywhere transverse to real-REL leaves and containing
$\bq_0$ such that:
\begin{enumerate}
\item
For any $\bq \in \mathcal{U}'$, $\psi_{\bq}$ is defined on $\mathcal{V}$. 
\item
The map $\mathcal{U}' \times \mathcal{V}$ defined by $(\bq, \cc)
\mapsto \psi_{\bq}(\cc)$ is a homeomorphism onto its image, which is
a neighborhood $\mathcal{U}$ of $\bq_0$. 
\item
Each plaque of the real REL foliation in $\mathcal{U}$ is of the form
$\psi_{\bq} (\mathcal{V})$. 

\end{enumerate}

\end{cor}

\begin{proof}
Let $\mathcal{U}_0$ be a bounded neighborhood of $\bq_0$ on which hol is
a local homeomorphism and let $\mathcal{U}' \subset \mathcal{U}_0$ be
any submanifold everywhere transverse to the real REL leaves and of
complementary dimension. For example we can take $\mathcal{U}'$ to be
the pre-image under hol of a small ball around $\hol(\bq_0)$ in an
affine subspace of $H^1(S, \Sigma; \R^2)$ which is complementary to
$W$. 

Since $\mathcal{U}_0$ is bounded there is a
uniform lower bound $r$ on the lengths of saddle connections for $\bq \in
\mathcal{U}_0$. If we let $\mathcal{V}_0$ be the ball of radius $r/2$
around 0 in
$W$, then the conditions of Theorem \ref{thm: real rel main} are 
satisfied for $\mathcal{V}_0$ and any $\bq \in
\mathcal{U}'$. Thus (1) holds for any $\mathcal{V} \subset
\mathcal{V}_0$. Taking for $\mathcal{V}$ a small ball around 0 we can assume that
$\Psi_{\bq}(\mathcal{V}) \subset \mathcal{U}_0$ for any
$\bq$. From \equ{eq: desired} and the choice of $\mathcal{U}'$ it
follows that $\{\psi_{\bq}(\cc): \bq \in \mathcal{U}', \cc \in
\mathcal{V}\}$ is a neighborhood of $\hol(\bq_0)$ in $H^1(S, \Sigma;
\R^2)$. Assertions (2),(3) now follow from the fact that
$\hol|_{\mathcal{U}_0}$ is a homeomorphism. 
\end{proof}

\ignore{

Suppose $\mathcal{K} \times \mathcal{V} \subset 
\mathcal{U} \times \mathcal{V}$ is compact. Since $\hol: \til \HH \to
H^1(S, \Sigma; \R^2)$ is continuous, the image $\hol(\mathcal{K})$ is
compact, so there are triangulations 
$\tau_1, \ldots, \tau_k$ of $(S, \Sigma)$ such that
$$\hol(\mathcal{U}_0 \times \mathcal{V}_0) \subset \bigcup_1^k \hol
(\til \HH_{\tau_i}).$$
Since $\hol|_{\til \HH_{\tau_i}}$ is a homeomorphism for each $i$, we
can choose

 $\bq_1$ is in the leaf of $\bq$. For this, we
repeat the above proof to construct, for any $s \in [0,1]$, a
translation surface $\bq(s)$ such that $\hol(\bq(s))-\hol(\bq)=s\cc$,
and such that the horizontal foliation of each $\bq(s)$ is $\GG$. For
each $s_0$ there is a neighborhood of $\bq(s_0)$ in $\til \HH$ in
which the real REL leaf is well-defined, and in particular there is
an open interval $I_{s_0}$ containing $s_0$ so that for $s \in [0,1] \cap
I_{s_0}$, there is $\bq'(s)$ in the same leaf as $\bq(s_0)$ such
that $\hol(\bq'(s)) - \hol(\bq(s_0)) = (s-s_0)\cc$. By Lemma \ref{lem: fiber
singleton}, we must 
have $\bq(s) = \bq'(s)$. Taking a finite subcover from the cover
$\{I_s\}$ of $[0,1]$ we see
that along the entire path, the surfaces $\bq(s)$ belong to the same
real REL leaf, as required. 
}

\begin{proof}[Proof of Theorem \ref{thm: real rel action}]
Let $\hat{\HH}, \, \til \HH$ be as above, let $\pi: \til \HH \to \hat{\HH}$ be the
projection, let $\til Q$ denote the subset of translation surfaces in $\til{\HH}$
without horizontal saddle connections, and let $Q =
\pi(\til Q)$. Note that $Q$ and $\til Q$ are $B$-invariant, where $B$
is the subgroup of upper triangular matrices in $G$, acting on $\hat{\HH}, \,
\til \HH$ in the usual way. We will extend the $B$-action to an action
of $F = B \ltimes W$. 

The action of $W$ is defined as follows. For each $\bq \in \til Q$,
the conditions of Theorem 
\ref{thm: real rel main} are vacuously satisfied for $\mathcal{V} =
W$, and we define $\cc \bq =
\psi_{\bq}(\cc)$. 
We first prove the group action law
\eq{eq: group law}{
(\cc_1+\cc_2)\bq = \cc_1(\cc_2
\bq)
} for the action of $W$. This follows from \equ{eq: desired}, associativity of
addition in $H^1(S, \Sigma; \R^2)$, and the uniqueness in Lemma
\ref{lem: fiber singleton}. 
Thus 
we have defined an
action of $W$ on $\til Q$, and by Proposition \ref{prop: independent of
marking} this descends to a well-defined action on $Q$.
To see that the action map $W \times
\til Q \to \til Q
$
is continuous, 
we take $w_n \to w$ in $W$, $\bq^{(n)} \to
\bq$ in $\til Q$, and need to show that 
\eq{eq: induction}{
w_n \bq^{(n)} \to  w\bq.
} 
Corollary \ref{cor:
  foliation compatible} implies that for any $t$ there are
neighborhoods $\mathcal{U}$ of $(tw)\bq$ and $\mathcal{V}$ of 0 in $W$
such that the map 
$$
\mathcal{U} \times \mathcal{V} \to \til \HH, \ \ (\bq, \cc) \mapsto
\cc \bq
$$
is continuous. By compactness we can find 
$\mathcal{U}_i, i=1, \ldots, k$, a partition $0 = t_0 < \cdots <
t_k = 1$, and a fixed open $\mathcal{V} \subset W$ such that
\begin{itemize}
\item
$\{(tw)\bq : t \in [t_{i-1}, t_i]\} \subset \mathcal{U}_i$.
\item
$(t_i -t_{i-1})w \in \mathcal{V}$. 
\end{itemize}

It now follows by induction on $i$ that 
$t_i w^{(n)} \bq_n \to (t_iw) \bq$ for each $i$, and putting $i=k$ we
get \equ{eq: induction}.

The action is
affine and measure preserving since in the
local charts given by $H^1(S, \Sigma; \R^2)$, it is defined by vector
addition. Since the area of a surface can be computed in terms of its
absolute periods alone, this action preserves the subset of area-one
surfaces. A simple 
calculation using \equ{eq: G action} shows that for $\bq \in \til Q, \,
\cc \in W$ and  
$$g = \left(\begin{matrix} a & b \\ 0 & a^{-1} \end{matrix} \right) \in
B, 
$$
we have
$g \cc \bq = (a \cc) g \bq$. That is, the actions of $B$ and $W$
respect the commutation relation defining $F$, so that we have defined
an $F$-action on $\til Q$. To check continuity of the action, let $f_n
\to f$ in $F$ and $\bq_n \to \bq \in \til Q$. Since $F$ is a
semi-direct product we can write $f_n = w_n b_n$, where $w_n \to w$
in $W$ and $b_n \to b$ in $B$. Since the $B$-action is
continuous, $b_n \bq_n \to b\bq$, and since (as verified above) the
$W$-action on $\til Q$ is continuous, 
$$f_n \bq_n = w_n (b_n \bq_n) \to
w (b \bq) = f\bq.$$ 
\end{proof}

\section{Cones of good directions}
\name{section: REL}
Suppose $\sigma$ is an irreducible and admissible permutation, and let
\eq{eq: defn cone}{
\CC_{\A} = \{\B \in \R^d: (\A, \B) \mathrm{\ is \ positive } \}
\subset T_{\A} \R^d_+.
}
As we showed in Corollary \ref{cor: same gamma}, this is the set of
good directions at the tangent space of $\A$, 
i.e. the directions of lines which may be lifted to horocycles. 
In this section we will relate the cones $\CC_{\A}$ 
with the bilinear form $Q$ as in \equ{eq: defn Q1}, and show there are
`universally good' directions for $\sigma$, i.e. specify
certain $\B$ such that $\B \in \CC_\A$ for  
all $\A$ which are without connections. We will also find `universally
bad' directions, i.e. directions which do not belong to $\CC_{\A}$ for
any $\A$; these will be seen to be related to real REL. 

Set
$\CC_{\A}^+ = \{\B \in \R^d: \mathrm{\equ{eq: positivity} \ holds} \}$, so
that $\CC_{\A} \subset \CC_{\A}^+$ for all $\A$, and $\CC_{\A}^+ = \CC_{\A}$ when
$\A$ has no connections. Now let
\eq{defn of C}{
\CC = \left\{\B \in \R^d: \forall i, \, Q(\E_i, \B)>0 \right \}. 
}
We have:
\begin{prop}
\name{prop: characterization of cone}
$\CC$ is a nonempty open convex cone, and 
$\CC = \bigcap_{\A} \CC_{\A}^+.$  

\end{prop}
\begin{proof}
It is clear that $\CC$ is open and convex. It follows from 
\equ{eq: relation L Q} that
$$\CC
= \left \{ \B \in \R^d: \forall x \in 
I, \, \forall \A \in \R^d_+, \, L_{\A,\B}(x)>0 \right \}
.
$$
The irreducibility of 
$\sigma$ implies that $\B_0=(b_1, \ldots, b_d)$ defined by 
$b_i = \sigma(i) - i$ (as in \cite{Masur-Keane})
satisfies $y_i(\B_0) > 0 > y'_i(\B_0)$ for all $i$, so 
by \equ{eq: defn L}, $L_{\A, \B_0}$
is everywhere positive irrespective of $\A$. 
This shows that $\B_0$ belongs to $\CC$, and moreover that $\CC$ is
contained in $\CC^+_{\A}$ for any $\A$. 

For the inclusion $\bigcap_{\A} \CC^+_{\A} \subset \CC$, 
suppose $\B \notin \CC_{\A}^+$, 
so that for some $\mu \in \MM_{\A}$ we have $\int L \, d\mu \leq 0.$
Writing $\A' = (a'_1, \ldots, a'_d)$, where $a'_j = \mu(I_j)$, we have
$$
Q(\A', \B) =\sum a'_i Q(\E_i, \B) =\int L \, d\mu \leq 0,$$
so that $\B
\notin \CC.$ 
\end{proof}

Note that in the course of the proof of Theorem \ref{thm: mahler
  curve}, we actually showed that
$\A'(s) \in -\CC$ for all $s>0$. Indeed, given a curve $\{\alpha(s)\}
\subset \R^d_+$, the easiest way to show that $\alpha(s)$ is uniquely
ergodic for a.e. $s$, is to show that $\alpha'(s) \in \pm \CC$
for a.e. $s$.

\medskip

Let $\mathbf{R}$ denote the null-space of $Q$, that is
$$\mathbf{R} = \{\B \in \R^d : Q(\cdot, \B) \equiv 0\} 
.$$

\begin{prop}
\name{prop: description REL}
$\mathbf{R} = \R^d \sm \bigcup_{\A} \pm \CC_{\A}.$
\end{prop}

\begin{proof}
By \equ{eq: relation L Q}, $\mathbf{R} = \{\B \in \R^d: L_{\A,\B} (x)
\equiv 0\},$ so that
containment $\subset$ is clear. Now suppose $\B \notin \mathbf{R}$, that
is $Q(\E_i, \B) \neq 0$ for some $i$, and by continuity there is an open
subset $\mathcal{U}$ of $\R^d_+$ such that $Q(\A, \B) \neq 0$ for $\A
\in \mathcal{U}$. Now taking $\A \in \mathcal{U}$ which is uniquely
ergodic and without connections we have $\B \in \pm \CC_{\A}.$ 
\end{proof}

Consider the map $(\A, \B) \mapsto \bq(\A, \B)$ 
defined in Theorem \ref{thm: sullivan2}. It is easy to see that the
image of an open subset of $\R^d_+ \times \{\B\}$ is a plaque for the real foliation. 
Additionally, recalling that $Q(\A, \B)$ records the intersection
pairing on $H_1(S, \Sigma) \times H_1(S \sm \Sigma)$, and that the
intersection pairing gives the duality $H_1(S\sm \Sigma) \cong H^1(S,
\Sigma)$, one finds that the image of $\mathbf{R} \times \{\B\}$ is a
plaque for the real REL foliation. That is, 
Proposition \ref{prop: description REL} says that the tangent
directions in $T\R^d_+$ which can never be realized as horocycle
directions, are precisely the directions in the real REL leaf. 

\combarak{Check that you agree. Along the way I had to use the
symmetry of $Q$ one more time, i.e. $\mathbf{R} = \{\A: Q(\A, \cdot)
\equiv 0\}$ and I am not too happy about it. This is related to my
question at the end of \S 2.4}

\ignore{
\section{Stable foliation for geodesics, horocycles, and real REL} 
This section contains some observations regarding the objects we have
considered, and some open questions. 

Recall that the action of $\{g_t\}$ on $\HH$ is called the {\em
geodesic flow}. Dynamically, the real foliation defined via 
\equ{eq: splitting} may be regarded as the unstable
foliation for this flow (see \cite{Veech geodesic flow}). This
foliation contains two sub-foliations

UNFINISHED

Let $\mathbf{R}$ denote the null-space of $Q$, that is
$$\mathbf{R} = \left\{\B \in \R^d : Q(\cdot, \B) \equiv 0 \right\} 
.$$

\begin{prop}
\name{prop: description REL}
$\mathbf{R} = \R^d \sm \bigcup_{\A} \pm \CC_{\A}.$
\end{prop}

\begin{proof}
We have $\mathbf{R} = \{\B \in \R^d: L_{\A,\B} (x) \equiv 0\}$
(as before the condition for $L_{\A, \B} \equiv 0$ does not depend on
$\A$), so that
containment $\subset$ is clear. Now suppose $\B \notin \mathbf{R}$, that
is $Q(\E_i, \B) \neq 0$ for some $i$, and by continuity there is an open
subset $\mathcal{U}$ of $\R^d_+$ such that $Q(\A, \B) \neq 0$ for $\A
\in \mathcal{U}$. Now taking $\A \in \mathcal{U}$ which is uniquely
ergodic and without connections we have $\B \in \pm \CC_{\A}.$ 
\end{proof}
}

\end{document}